\documentclass{amsart}
\usepackage{amsmath}
\usepackage{amsfonts}
\usepackage{amssymb}
\usepackage{graphicx}
\usepackage{color}
\usepackage{amsmath}
\usepackage{placeins}

\usepackage{enumitem}

\usepackage{amsfonts}
\usepackage{amssymb}
\usepackage{graphicx}
\usepackage{xypic}
\usepackage{tikz-cd}
\usepackage{color}
\usepackage{mathbbol}
\usepackage{bbold}

\usepackage{tikz,tikz-cd}
\usetikzlibrary{matrix, positioning, decorations.pathreplacing, shapes,automata,arrows,calc}
\usepackage{caption}
\usepackage{subcaption}

\providecommand{\U}[1]{\protect\rule{.1in}{.1in}}
\newtheorem{theorem}{Theorem}[section]

\newtheorem{corollary}[theorem]{Corollary}

\newtheorem{definition}[theorem]{Definition}

\newtheorem{lemma}[theorem]{Lemma}

\newtheorem{remark}[theorem]{Remark}

\newcommand\nonumberthis{\nonumber\refstepcounter{equation}}

\numberwithin{equation}{section}

\title{Decomposition theorems for unital graph C*-algebras}
\author{Guillaume Bellier, Tatiana Shulman}

\begin{document}

\maketitle

\begin{abstract}  We prove that unital graph C*-algebras often admit a convenient decomposition into amalgamated free products. We use this to give a complete characterization of when a unital graph C*-algebra is residually finite-dimensional  and when it is operator norm stable (that is, matricially semiprojective).
\end{abstract}

\section{Introduction}

The class of graph C*-algebras is prominent in C*-theory, and numerous works are dedicated to studying how C*-algebraic properties of a graph C*-algebra are related with properties of the graph. In this work we characterize, for the class of all unital graph C*-algebras,  two properties that are of central importance in C*-theory:  residual finite-dimensionality (RFD property) and  operator norm stability (also called matricial semiprojectivity).

 As a crucial ingredient, we prove that graph C*-algebras of many graphs can be decomposed as amalgamated free products, a result that can be of independent interest \footnote{Such decomposition should not  be confused with push-ups of graphs and corresponding pull-backs for graph algebras  from \cite{Tobolski} -- a construction differing from the amalgamated free products,  in which all the arrows are reversed compared to those in the definition of amalgamated free product.}.

Recall that a graph C*-algebra is unital if and only the graph has finitely many vertices (and possible infinitely many edges).

 \bigskip

{\bf Theorem} (Theorems 3.2 and 3.3) \; {\it Let  $G = G_1\bigcup G_2$ be a graph with finitely many vertices such that  no edge of $G_2$ enters $G_1$. Let $n$ be the number of vertices shared by $G_1$ and $G_2$. Then

1) If $G^0_2 = G^0$, then $$C^*(G) \cong \left(C^*(G_1)\oplus \mathbb C\right) \ast_{ \mathbb C^{n+1}} C^*(G_2),$$

2) If $G^0_2 \neq G^0$, then  $$C^*(G) \cong \left(C^*(G_1)\oplus \mathbb C\right) \ast_{ \mathbb C^{n+2}} \left(C^*(G_2)\oplus \mathbb C\right).$$}


\medskip

A C*-algebra is {\it residually finite-dimensional} (RFD) if it has a separating family of finite-dimensional representations.
The property of being RFD plays an important role in several long-standing problems in operator algebras. E.g. Kirchberg's conjecture or,  equivalently, the Connes Embedding Problem, states that $C^*(F_2\times F_2)$ is RFD (see \cite{Taka})\footnote{Connes Embedding Problem was resolved in the negative in \cite{CEP}}. The question of when group C*-algebras  are RFD has useful implications in group theory.   On the purely C*-theoretic side,  the RFD property is important in the developments on the UCT conjecture \cite{Dadarlat1} and the problem of finding subalgebras of AF-algebras \cite{Dadarlat2}.
 Over the years, various
characterizations of RFD C*-algebras have been obtained,
and numerous classes of C*-algebras - too many to list -  have been shown to be RFD.

 The following theorem gives a complete characterization of when a unital graph C*-algebra is RFD.
 \bigskip

 {\bf Theorem } (Theorem 4.3) {\it Let $G$ be a graph a finitely many vertices. Then $C^*(G)$ is RFD if and only if no cycle has an entry.}

 \bigskip

 It is interesting that for arbitrary graphs, the property that no cycle has an entry is equivalent to quasidiagonality as proved in \cite{Schafhauser}. Quasidiagonality is an approximation property of C*-algebras that is strictly weaker than RFD, but for graph C*-algebras of finite graphs our theorem above shows that these properties coincide.


\medskip

Operator norm stability (also known as matricial stability or matricial semiprojectivity) is the property that every finite-dimensional approximate representation is close, in operator norm, to a genuine finite-dimensional representation. This notion has been studied for over fifty years and has found significant applications in Elliott's classification program for simple nuclear C*-algebras, as well as in K-theory and E-theory for C*-algebras. In the past decade, matricial semiprojectivity has also become a central topic in group theory, due to its connections with approximation conjectures for groups.

In \cite{EilersKatsura} Eilers and Katsura gave a complete description of which unital graph C*-algebras
are semiprojective and weakly semiprojective (see Preliminaries section for definitions). Here we give a complete description of which unital graph C*-algebras
are matricially semiprojective.

To formulate the result, we first consider a subgraph $G'$ of $G$ consisting of all paths that lead to cycles. Then we define a subgraph $\tilde G$ of $G$ in terms of reachability from $G'$, see the beginning of section 5 for precise definition of $\tilde G$.
\bigskip

{\bf Theorem } (Theorem 5.14) {\it  Let $G$ be a graph with finitely many vertices. Then $C^*(G)$ is matricially semiprojective if and only if $\tilde G$ is finite.}

\bigskip

\medskip

\noindent {\bf Acknowledgments. } The authors would like to express their sincere gratitude to Adam Dor-On for his careful reading of an earlier version of this paper and for his many insightful comments and suggestions, which significantly improved the exposition of the paper.

  The second named author was partially supported by a grant from the Swedish Research Council.

  \medskip

  The results of this article are part of the PhD project of the first
named author.

\section{Preliminaries}

\subsection{Graph C*-algebras}

A directed graph $G = (G^0, G^1, s, r)$  consists of two countable sets $G^0$ and $G^1$ and functions $r, s: G^1\to G^0$. The elements
of $G^0$ are called {\it vertices} and the elements of $G^1$ are called {\it edges}. For each edge $e$, $s(e)$ is the {\it source} of $e$ and $r(e)$ is the {\it range } of $e$.

\medskip

The {\it graph C*-algebra} $C^*(G)$ associated with a graph $G$  is the universal $C^*$-algebra with generators $\{p_v, s_e\;|\; v\in G^0, e\in G^1\}$, where $\{p_v, v\in G^0\}$ is a collection of mutually orthogonal projections, $\{s_e, e\in G^1\}$ is
 a collection of partial isometries with mutually orthogonal ranges, and the following three
 relations are satisfied:

\medskip

\begin{enumerate}





\item $s_e^*s_e = p_{s(e)}$, for each  $e\in G^1$,

\medskip

\item $p_v = \sum_{e\in r^{-1}(v)} s_es_e^*$, for each $v\in G^0$ such that $0< \sharp(r^{-1}(v)) < \infty$,

\medskip

\item $s_es_e^* \le p_{r(e)}$, for each  $e\in G^1$.

\end{enumerate}

\medskip

 The relations above are {\it Cuntz-Krieger relations} (CK-relations, for short).

$C^*(G)$ is unital precisely when $G^0$ is finite. In this case $\sum_{v\in G^0}p_v$ is the unit of $C^*(G)$.

Throughout this paper we will use same notation  (usually, p) for a vertex and the projection associated with it, and we will use  same notation (usually, e) for an edge and the partial isometry associated with it.

\medskip

A {\it path} in $G$ is a sequence of edges $\nu = (\nu_n, \ldots, \nu_1)$ such that $s(\nu_{i+1}) = r(\nu_i)$ for each $1\le i\le n$.  A {\it cycle} is a path $\mu = (\mu_n, \ldots, \mu_1)$ such that $s(\mu_1) = r(\mu_n)$.

Let $G^*$ be the set of all paths on $G$.

For a vertex $t$ let $n(t)$ be the number of paths that start at $t$, that is  $$n(t) = \sharp\{\nu\in G^*\;|\; s(\nu) = t\}.$$

We say that a vertex $t$ is a {\it source} if it does not receive any edges, that is  $r^{-1}(t) = \emptyset$.

We say that  {\it no cycle has an entry} meaning that no edge enters a cycle (besides the edges of the cycle itself, of course).

We will need the following two theorems about graph C*-algebras.

\begin{theorem}\label{tree}(\cite[Prop. 1.18]{Raeburn}, \cite[Th. 2.1.3]{Tomforde}\footnote{The terminology in \cite{Tomforde}  is different from \cite{Raeburn}, in particular what is a source in \cite{Raeburn} is a sink in \cite{Tomforde}.}) If $G$ is a finite graph with no cycles, and $v_1, \ldots, v_m$ are the sources of $G$, then
  $$C^*(G) \cong \bigoplus_{i=1}^m M_{n(v_i)}.$$ This isomorphism sends each source to a rank one projection.\footnote{This follows from the construction of the isomorphism. }
\end{theorem}

\begin{theorem}(\cite[Th. 1.4.14]{Tomforde}, \cite[Ex.2.14]{Raeburn}) If $G$ is the graph consisting of a single cycle with $n$ vertices, then there is an isomorphism $\rho: C^*(G)\to M_n\otimes C(\mathbb T)$ given by
$$\rho(s_{e_i}) := \begin{cases} E_{(i+1)i}\otimes 1, \; \text{if} \; 1\le i\le n-1\\
E_{1n}\otimes id_{C(\mathbb T)}, \;\text{if}\; i=n,\end{cases}$$

$$\rho(p_{v_i}):= E_{ii}\otimes 1 \; \text{for} \; 1\le i\le n,$$

\noindent where $\{E_{ij} \}_{i, j=1}^n$ is the matrix unit in $M_n$.
\end{theorem}

In particular from the last theorem we observe the  following.

\begin{corollary}\label{cycle} Let $\rho$ be as in the theorem above. For $z\in \mathbb T$ let $\rho_z = ev_z\circ \rho$ (here we identify elements of $M_n\otimes C(\mathbb T)$ with $M_n$-valued continuous functions on $\mathbb T$). Then

\begin{enumerate}

\item for each $z\in \mathbb T$, $\rho_z$ is a representation that sends each $p_v$ to a rank one projection,

\item $\rho_z(p_v)$ does not depend on $z$,

\item for a dense subset $\mathcal F$ of $\mathbb T$, $\oplus_{z\in \mathbb T} \rho_z$ is injective.
\end{enumerate}
\end{corollary}

\subsection{Amalgamated free products.}

\begin{definition} Let $A, B, D$ be unital C*-algebras with unital embeddings $\theta_A: D \to A$ and $\theta_B: D \to B$. The unital full amalgamated free product of $A$ and $B$ over $D$ is the C*-algebra $C$, equipped with unital embeddings $i_A: A \to C$ and $i_B: B\to C$ satisfying $i_A\circ \theta_A = i_B\circ \theta_B$, such that $C$ is generated by $i_A(A)\bigcup i_B(B)$ and satisfies the following universal property:

whenever $\mathcal E$ is a unital C*-algebra and $\pi_A: A \to \mathcal E$ and $\pi_B: B \to \mathcal E$ are unital $\ast$-homomorphisms satisfying $\pi_A\circ\theta_A = \pi_B\circ \theta_B$, there is  a unital  $\ast$-homomorphism $\pi: C \to \mathcal E$ such that $\pi\circ i_A = \pi_A$ and $\pi\circ i_B = \pi_B$.
\end{definition}

Standardly the (unital) amalgamated free product $C$ is denoted by $A\ast_D B$.

\medskip

The following theorem of Li and Shen  gives a necessary and sufficient condition for a unital  amalgamated free product over a finite-dimensional C*-subalgebra to be RFD. A different proof and a non-unital version are obtained in \cite[Th. 4.14]{EndersShulman}.

\begin{theorem}\label{LiShen}(Li and Shen \cite{LiShen}) Let $A$ and $B$ be unital C*-algebras, $F$ a finite-dimensional C*-algebra and $\theta_A: F \to A$, $\theta_B: F \to B$ unital inclusions. Then the corresponding unital amalgamated free product $A\ast_F B$ is RFD if and only if there exist unital inclusions $\phi_A: A\to \prod M_n, \phi_B: B \to \prod M_n$ such that $\phi_A\circ \theta_A = \phi_B\circ \theta_B$.
\end{theorem}

Throughout this paper all amalgamated free products will be unital ones.

\subsection{Matricial semiprojectivity}

\begin{definition}  A $C^*$-algebra $A$ is {\bf semiprojective}  if for any $C^*$-algebra $A$, any increasing chain of ideals $I_n$
in $A$  and every $\ast$-homomorphism $\varphi:A \to B/I$, where   $I=\overline{\bigcup I_n}$, there exist
 $n\in \mathbb N$ and  a $\ast$-homomorphism $\overline \varphi : A\to B/I_n$ making the following diagram commute:
$$\xymatrix{ & B \ar@{->>}[d]\\& B/I_n \ar@{->>}[d] \\ A\ar@{-->}[ru]^{\overline \varphi}\ar[r]^{\varphi} & B/I}$$
\end{definition}

\bigskip

For C*-algebras $B_n$, $n\in \mathbb N$, let us consider the following two C*-algebras of sequences with entries in $B_n$, $n\in \mathbb N$:
$$\prod B_n = \{(b_n)\;|\; b_n \in B_n, \; \sup \|b_n\| < \infty\},$$
$$\bigoplus B_n = \{(b_n)\in \prod B_n\;|\; \lim \|b_n\| = 0\}.$$ Clearly $\bigoplus B_n$ is  an ideal in $\prod B_n$.

\begin{definition} Let $\mathcal B$ be a class of C*-algebras. A C*-algebra $A$ is {\bf weakly semiprojective with respect to $\mathcal B$} if for any sequence $B_n\in \mathcal B$, any $\ast$-homomorphism  $\varphi: A\to \prod B_k/\bigoplus B_k$ lifts to a $\ast$-homomorphism  $\overline{\varphi}: A\to \prod B_k$ with $\varphi = \pi\circ \overline{\varphi}.$
\end{definition}
If the class $\mathcal B$ above is the class of all $C^*$-algebras then we say that $A$ is {\bf weakly semiprojective}.

Let $M_n$ denote the C*-algebra  of all $n\times n$-matrices with complex entries.

\begin{definition}\label{def msp}
A separable $C^*$-algebra $A$ is {\bf matricially semiprojective} if one of the following equivalent conditions holds:
\begin{enumerate}
\item Any $\ast$-homomorphism
$$\varphi\colon A\to\prod M_{n} / \bigoplus M_{n}$$
lifts to a $\ast$-homomorphism $\overline{\varphi}\colon A\to\prod M_{n}$ with $\pi\circ\overline{\varphi}=\varphi$.
\item For any sequence $F_n$ of finite-dimensional $C^*$-algebras, any $\ast$-homomorphism
$$\varphi\colon A\to\prod F_n / \bigoplus F_n$$
lifts to a $\ast$-homomorphism $\overline{\varphi}\colon A\to\prod F_n$ with $\pi\circ\overline{\varphi}=\varphi$, i.e.\ $A$ is weakly semiprojective with respect to the class of finite-dimensional $C^*$-algebras.
\end{enumerate}
\end{definition}

Of course semiprojectivity implies weak semiprojectivity, and weak semiprojectivity implies matricial semiprojectivity.

\section{Decomposition theorems}

Suppose $G= G_1\bigcup G_2$ with non-empty $G_1, G_2$. We assume that $G_1$ and $G_2$ share only some vertices but not edges.

The next lemma contains an observation which will be crucial for the rest of this section.

\begin{lemma}\label{crucial} Suppose $G=G_1\bigcup G_2$ and no edge of $G_2$ enters $G_1$. Then
$$\{\text{CK-relations of} \;G\} = \{\text{CK-relations of} \;G_1\}\; \bigcup \;\{\text{CK-relations of} \;G_2\}.$$
\end{lemma}
\begin{proof} When $G=G_1\bigcup G_2$, the only CK-relations of $G$ that could possibly contain edges from both $G_1$ and $G_2$
are of the form $$p_i = \sum_{e\in G^1, r(e) = p_i} ee^*,$$ where $p_i$ is one of the shared vertices. But since in our case
no edge of $G_2$ enters $G_1$, the condition $r(e) = p_i$ implies that $e\in G_1^1$. Therefore
$$p_i = \sum_{e\in G^1, r(e) = p_i} ee^* = \sum_{e\in G_1^1, r(e) = p_i} ee^*$$ is a CK-relation for $G_1$.
\end{proof}

\medskip

When dealing with $G=G_1\bigcup G_2$,  we will use the following notation.
For a vertex $p$ in $G_1$, the corresponding projection in $C^*(G)$ is also denoted by $p$ as already was mentioned in Preliminaries, and the corresponding projection in $C^*(G_1)$ will be denoted by $\bar p$.

 For a vertex $p$ in $G_2$, the corresponding projection in $C^*(G)$ is denoted by $p$, and the corresponding projection in $C^*(G_2)$ is denoted by $\bar{\bar p}$. Similar notation we use for partial isometries corresponding to the edges.

\medskip

Common vertices of $G_1$ and $G_2$ will be denoted by $p_i$. Vertices in $G_1\setminus G_2$ will be denoted by $p_{\alpha}$. Vertices in $G_2\setminus G_1$ will be denoted by $p_{\beta}$.

\bigskip

We will need to consider separately the two possible cases:     $G^0_2 = G^0$
  and
  $\nolinebreak{G^0_2 \neq G^0}$. Examples are given in Figure \ref{exemple}.

\tikzset{every loop/.style={min distance=10mm,in=120,out=240,looseness=40}}
\begin{figure}
\centering
\begin{subfigure}[t]{0.4\textwidth}
\centering
 \begin{tikzpicture}[node distance={15mm}, thick, main/.style = {draw, fill, circle, inner sep=1.5pt},mainLarge/.style = {draw, circle, minimum size=30mm}]
 \useasboundingbox (-2,-0.5) rectangle (4,2.5);

 \def\Unit{0.5}

\node[main, label=below:$ $,black] (P1) at (0*\Unit,2*\Unit){};
\node[main, label=below:$ $,blue] (P2) at (2*\Unit,4*\Unit){};
\node[main, label=below:$ $,blue] (P3) at (2*\Unit,2*\Unit){};
\node[main, label=below:$ $,blue] (P4) at (2*\Unit,0*\Unit){};

\path[-angle 90,font=\scriptsize]
(P1) edge [loop left,"",red]   (P1)

(P1) edge ["",blue]   (P2)
(P1) edge ["",blue]   (P3)
(P1) edge ["",blue]   (P4);

\end{tikzpicture}
   \caption*{Case $G^0_2 = G^0$}
\end{subfigure}
\hfill
\begin{subfigure}[t]{0.4\textwidth}
\centering

\begin{tikzpicture}[node distance={15mm}, thick, main/.style = {draw, fill, circle, inner sep=1.5pt},mainLarge/.style = {draw, circle, minimum size=30mm}]
\useasboundingbox (-0.5,-0.5) rectangle (4.5,2.5);

\def\Unit{0.5}

\node[main, label=below:$ $,red] (P1) at (0*\Unit,1*\Unit){};
\node[main, label=below:$ $,red] (P2) at (2*\Unit,0*\Unit){};
\node[main, label=below:$ $,black] (P3) at (4*\Unit,1*\Unit){};
\node[main, label=below:$ $,black] (P4) at (4*\Unit,3*\Unit){};
\node[main, label=below:$ $,red] (P5) at (2*\Unit,4*\Unit){};
\node[main, label=below:$ $,red] (P9) at (0*\Unit,3*\Unit){};
\node[main, label=below:$ $,blue] (P6) at (8*\Unit,4*\Unit){};
\node[main, label=below:$ $,blue] (P7) at (8*\Unit,2*\Unit){};
\node[main, label=below:$ $,blue] (P8) at (8*\Unit,0*\Unit){};

\path[-angle 90,font=\scriptsize]
(P1) edge ["",red]   (P2)
(P2) edge ["",red]   (P3)
(P3) edge ["",red]   (P4)
(P4) edge ["",red]   (P5)
(P5) edge ["",red]   (P9)
(P9) edge ["",red]   (P1)

(P4) edge ["",blue]   (P6)
(P3) edge ["",blue]   (P6)
(P3) edge ["",blue]   (P7)
(P3) edge ["",blue]   (P8);
\end{tikzpicture}
 \caption*{Case $G^0_2 \neq G^0$}
\end{subfigure}
\caption{Examples}
\label{exemple}
\end{figure}



\subsection{Case $G^0_2 = G^0$.}

We can assume that $G^0_1 \neq G^0$. Indeed, if $G^0_1 = G^0$, then, since we assume that no edge of $G_2$ enters $G_1$, we have $G_1=G$ and $G_2$ is just the collection of all vertices of $G$ with no edges at all, so this case is trivial. 

\begin{theorem}\label{case1}  Let  $G = G_1\bigcup G_2$ be a graph with finitely many vertices such that $G^0_2 = G^0$, $G^0_1 \neq G^0$,  and no edge of $G_2$ enters $G_1$. Let $n$ be the number of vertices shared by $G_1$ and $G_2$. Then $$C^*(G) \cong \left(C^*(G_1)\oplus \mathbb C\right) \ast_{ \mathbb C^{n+1}} C^*(G_2),$$ where amalgamation  is done via the unital embeddings $\theta_1: \mathbb C^{n+1} \to C^*(G_1)\oplus \mathbb C$ and $\theta_2: \mathbb C^{n+1} \to C^*(G_2)$ defined by
$$\theta_1(\lambda_1, \ldots, \lambda_{n+1}) = (\sum_{i=1}^n \lambda_i\bar p_i, \lambda_{n+1}),$$
$$\theta_2(\lambda_1, \ldots, \lambda_{n+1}) = \sum_{i=1}^n \lambda_i\bar{\bar p}_i + \lambda_{n+1}\sum_{\beta} \bar{\bar p}_{\beta},$$
$\lambda_1, \ldots, \lambda_{n+1} \in \mathbb C$.
\end{theorem}
\begin{proof} First we will construct a unital $\ast$-homomorphism $\phi: \left(C^*(G_1)\oplus \mathbb C\right) \ast_{ \mathbb C^{n+1}} C^*(G_2) \to C^*(G)$. For that we at first define a (non-unital) $\ast$-homomorphism $\phi_0: C^*(G_1) \to C^*(G)$ by
$$\phi_0(\bar p) = p, \; \phi_0(\bar e) = e,$$ for $\bar p\in G^0_1, \bar e\in G_1^1$. Lemma \ref{crucial} implies that $\phi_0$ is well-defined. Since there is no edge of $G_2$ entering $G_1$,  $\{p_{\beta}\} \neq \emptyset$.
Since $\sum_{p\in G^0_1} p + \sum_{\beta} p_{\beta} = 1_{C^*(G)}, $ we can define a unital $\ast$-homomorphism $$\phi_1:   C^*(G_1)\oplus \mathbb C \to C^*(G)$$ by
$$\phi_1\left((a, \lambda)\right) = \phi_0(a) + \lambda \sum_{\beta} p_{\beta},$$
for any $a\in C^*(G_1), \lambda \in \mathbb C$.

We define a $\ast$-homomorphism $\phi_2: C^*(G_2) \to C^*(G)$ by
$$\phi_2(\bar{\bar p}) = p, \; \phi_2(\bar {\bar e}) = e,$$
for  $\bar{\bar p}\in G^0_2, \bar{\bar e}\in G_2^1$. Lemma \ref{crucial} implies that $\phi_2$ is well-defined. Our assumption that $G^0_2 = G^0$ implies that $\phi_2$ is unital.
  We have
$$\phi_1\circ\theta_1 (\lambda_1, \ldots, \lambda_{n+1}) = \phi_1\left((\sum_{i} \lambda_i\bar p_i, \lambda_{n+1})\right)=
\sum_i \lambda_i p_i + \lambda_{n+1} \sum_{\beta} p_{\beta},$$
$$ \phi_2\circ\theta_2 (\lambda_1, \ldots, \lambda_{n+1}) = \phi_2\left(\sum_{i=1}^n \lambda_i\bar{\bar p}_i + \lambda_{n+1}\sum_{\beta} \bar{\bar p}_{\beta}\right) = \sum_{i=1}^n \lambda_i p_i + \lambda_{n+1}\sum_{\beta} p_{\beta},  $$ for any $(\lambda_1, \ldots, \lambda_{n+1})\in \mathbb C^{n+1}.$
So $\phi_1\circ\theta_1 = \phi_2\circ\theta_2$ and therefore $\phi_1, \phi_2$ give rise to a unital $\ast$-homomorphism $\phi: \left(C^*(G_1)\oplus \mathbb C\right) \ast_{ \mathbb C^{n+1}} C^*(G_2) \to C^*(G)$ such that $\phi\circ i_i = \phi_i, $ for $i=1, 2.$

Now we define a $\ast$-homomorphism $\psi: C^*(G) \to \left(C^*(G_1)\oplus \mathbb C\right) \ast_{ \mathbb C^{n+1}} C^*(G_2)$ by

$$\psi(p) = \begin{cases} i_1\left((\bar p, 0)\right), p\in G_1 \\ i_2(\bar{\bar p}), p\in G_2, \end{cases}$$
$$\psi(e) = \begin{cases} i_1\left((\bar e, 0)\right), e\in G_1 \\ i_2(\bar{\bar e}), e\in G_2 \end{cases}$$
(we used here the universal property of $C^*(G)$).
Let $\delta_i$ denote the element of $\mathbb C^{n+1}$ whose $i$-th coordinate is one and all the other coordinates are zero.
Since $$i_1\left((\bar p_i, 0)\right) = i_1\circ\theta_1(\delta_i) = i_2\circ\theta_2(\delta_i) =i_2\left(\bar{\bar p}_i\right),$$ $\psi$ is well-defined on the set of generators of $C^*(G)$ and therefore by Lemma \ref{crucial} on the whole $C^*(G)$. Let us check that $\psi$ is unital. Since $G^0_2 = G^0$,
\begin{multline*}\psi(1_{C^*(G)}) = \psi\left(\sum_i p_i + \sum_{\beta}p_{\beta}\right) = i_2\left(\sum_i \bar{\bar p}_i + \sum_{\beta} \bar{\bar p}_{\beta}\right)\\  =i_2(1_{C^*(G_2)}) = 1_{\left(C^*(G_1)\oplus \mathbb C\right) \ast_{ \mathbb C^{n+1}} C^*(G_2)}.\end{multline*}

It is straightforward to check that $\phi\circ \psi = id_{C^*(G)}$. To show that $\psi\circ \phi = id_{\left(C^*(G_1)\oplus \mathbb C\right) \ast_{ \mathbb C^{n+1}} C^*(G_2)}$, it is sufficient to verify this on the generators $i_1(\bar p, 0)$, $i_1(\bar e, 0), i_1((0, 1))$, $i_2(\bar{\bar p})$, $i_2(\bar{\bar e})$ of $\left(C^*(G_1)\oplus \mathbb C\right) \ast_{ \mathbb C^{n+1}} C^*(G_2)$. For all the generators but $i_1(0, 1)$ it is straightforward. For $i_1((0, 1))$ we verify
\begin{multline*}\psi\circ\phi\left(i_1(0, 1)\right) = \psi(\phi_1((0, 1))) = \psi(\sum_{\beta} p_{\beta}) = i_2(\sum_{\beta} \bar{\bar p}_{\beta})\\
= i_2\circ\theta_2(\delta_{n+1}) = i_1\circ\theta_1(\delta_{n+1}) = i_1\left((0, 1)\right). \end{multline*}

\end{proof}

\subsection{Case $G^0_2 \neq G^0$.}

\begin{theorem}\label{case2}  Let  $G = G_1\bigcup G_2$ be a graph with finitely many vertices such that $G^0_2 \neq G^0$ and no edge of $G_2$ enters $G_1$. Let $n$ be the number of vertices shared by $G_1$ and $G_2$. Then $$C^*(G) \cong \left(C^*(G_1)\oplus \mathbb C\right) \ast_{ \mathbb C^{n+2}} \left(C^*(G_2)\oplus \mathbb C\right),$$ where amalgamation  is done via the unital embeddings $\theta_1: \mathbb C^{n+2} \to C^*(G_1)\oplus \mathbb C$ and $\theta_2: \mathbb C^{n+2} \to C^*(G_2)\oplus \mathbb C$ defined by
$$\theta_1(\lambda_1, \ldots, \lambda_{n+2}) = (\sum_{i=1}^n \lambda_i\bar p_i + \lambda_{n+1}\sum_{\alpha} \bar p_{\alpha}, \lambda_{n+2}),$$
$$\theta_2(\lambda_1, \ldots, \lambda_{n+2}) = (\sum_{i=1}^n \lambda_i\bar{\bar p}_i + \lambda_{n+2}\sum_{\beta} \bar{\bar p}_{\beta}, \lambda_{n+1}),$$
$\lambda_1, \ldots, \lambda_{n+2} \in \mathbb C$.
\end{theorem}
\begin{proof}
First we will construct a  unital $\ast$-homomorphism $\phi: \left(C^*(G_1)\oplus \mathbb C\right) \ast_{ \mathbb C^{n+2}} \left(C^*(G_2)\oplus \mathbb C\right) \to C^*(G)$. For that we at first define a (non-unital) $\ast$-homomorphism $\phi_{1,0}: C^*(G_1) \to C^*(G)$ by
$$\phi_{1,0}(\bar p) = p, \; \phi_{1,0}(\bar e) = e,$$ for $\bar p\in G^0_1, \bar e\in G_1^1$. Lemma \ref{crucial} implies that $\phi_{1,0}$ is well-defined.
Since $\sum_{p\in G^0_1} p + \sum_{\beta} p_{\beta} = 1_{C^*(G)}, $ we can define a unital $\ast$-homomorphism $$\phi_1:   C^*(G_1)\oplus \mathbb C \to C^*(G)$$ by
$$\phi_1\left((a, \lambda)\right) = \phi_{1,0}(a) + \lambda \sum_{\beta} p_{\beta},$$
for any $a\in C^*(G_1), \lambda \in \mathbb C$. We note that $\{p_{\beta}\} \neq \emptyset$ because otherwise there would be an edge in $G_2$ entering $G_1$. Therefore $\phi_1$ is well-defined.

We define a $\ast$-homomorphism $\phi_2: C^*(G_2)\oplus\mathbb C \to C^*(G)$ analogously, that is
we define  a $\ast$-homomorphism $\phi_{2,0}: C^*(G_2) \to C^*(G)$ by
$$\phi_{2,0}(\bar{\bar  p}) = p, \; \phi_{2,0}(\bar{\bar  e}) = e,$$ for $\bar{\bar  p}\in G^0_2, \bar{\bar e}\in G_2^1$
and a unital $\ast$-homomorphism $$\phi_2: C^*(G_2)\oplus \mathbb C \to C^*(G)$$ by
$$\phi_2\left((a, \lambda)\right) = \phi_{2,0}(a) + \lambda \sum_{\alpha} p_{\alpha},$$
for any $a\in C^*(G_2), \lambda \in \mathbb C$. We note that since $G^0_2\neq G^0$,  there must be a vertex in $G_1$ that does not belong to $G_2$. So $\{p_{\alpha}\} \neq \emptyset$, and $\phi_2$ is well-defined.

  We have
\begin{multline*}\phi_1\circ\theta_1 (\lambda_1, \ldots, \lambda_{n+2}) = \phi_1\left((\sum_{i} \lambda_i\bar p_i + \lambda_{n+1} \sum_{\alpha} \bar p_{\alpha}, \lambda_{n+2})\right)\\ =
\sum_i \lambda_i p_i + \lambda_{n+1} \sum_{\alpha} p_{\alpha} + \lambda_{n+2} \sum_{\beta} p_{\beta},\end{multline*}
 \begin{multline*}\phi_2\circ\theta_2 (\lambda_1, \ldots, \lambda_{n+2}) = \phi_2\left((\sum_{i=1}^n \lambda_i\bar{\bar p}_i + \lambda_{n+2}\sum_{\beta} \bar{\bar p}_{\beta}, \lambda_{n+1})\right) \\ = \sum_{i=1}^n \lambda_i p_i + \lambda_{n+2}\sum_{\beta} p_{\beta} + \lambda_{n+1}\sum_{\alpha} p_{\alpha},  \end{multline*} for any $(\lambda_1, \ldots, \lambda_{n+2})\in \mathbb C^{n+2}.$
So $\phi_1\circ\theta_1 = \phi_2\circ\theta_2$ and therefore $\phi_1, \phi_2$ define a unital $\ast$-homomorphism $\phi: \left(C^*(G_1)\oplus \mathbb C\right) \ast_{ \mathbb C^{n+2}} \left(C^*(G_2)\oplus \mathbb C\right) \to C^*(G)$ such that $\phi\circ i_i = \phi_i, $ for $i=1, 2.$

Now we define a $\ast$-homomorphism $\psi: C^*(G) \to \left(C^*(G_1)\oplus \mathbb C\right) \ast_{ \mathbb C^{n+2}} \left(C^*(G_2)\oplus \mathbb C\right)$ by

$$\psi(p) = \begin{cases} i_1\left((\bar p, 0)\right), p\in G_1 \\ i_2((\bar{\bar p}, 0)), p\in G_2, \end{cases}$$
$$\psi(e) = \begin{cases} i_1\left((\bar e, 0)\right), e\in G_1 \\ i_2((\bar{\bar e}, 0)), e\in G_2. \end{cases}$$
Let $\delta_i$ denote the element of $\mathbb C^{n+2}$ whose $i$-th coordinate is one and all the other coordinates are zero.
Since $$i_1\left((\bar p_i, 0)\right) = i_1\circ\theta_1(\delta_i) = i_2\circ\theta_2(\delta_i) =i_2\left((\bar{\bar p}_i, 0)\right),$$ $\psi$ is well-defined on the set of generators of $C^*(G)$ and therefore by Lemma \ref{crucial} on the whole $C^*(G)$. Let us check that $\psi$ is unital.

\begin{multline*} \psi(1_{C^*(g)}) = \psi(\sum_i p_i + \sum_{\alpha} p_{\alpha} + \sum_{\beta} p_{\beta}) =
\sum_i i_1((\bar p_i, 0))  + \sum_{\alpha} i_1((\bar p_{\alpha}, 0)) + \sum_{\beta} i_2((\bar{\bar  p}_{\beta}, 0))\\
= \sum_i i_1(\theta_1(\delta_i)) + i_1(\theta_1(\delta_{n+1})) + i_2(\theta_2(\delta_{n+2})) \\ = \sum_i i_1(\theta_1(\delta_i)) + i_1(\theta_1(\delta_{n+1})) + i_1(\theta_1(\delta_{n+2})) = i_1(\theta_1(1_{\mathbb C^{n+2}})) = 1_{\left(C^*(G_1)\oplus \mathbb C\right) \ast_{ \mathbb C^{n+2}} \left(C^*(G_2)\oplus \mathbb C\right)}.
\end{multline*}

It is straightforward to check that $\phi\circ \psi = id_{C^*(G)}$. To show that $\psi\circ \phi = id_{\left(C^*(G_1)\oplus \mathbb C\right) \ast_{ \mathbb C^{n+2}} \left(C^*(G_2)\oplus \mathbb C\right)}$, it is sufficient to verify this on the generators $i_1(\bar p, 0)$, $i_1(\bar e, 0)$, $i_1((0, 1))$, $i_2((\bar{\bar p}, 0))$, $i_2((\bar{\bar e}, 0))$, $i_2((0, 1))$ of $\left(C^*(G_1)\oplus \mathbb C\right) \ast_{ \mathbb C^{n+1}} C^*(G_2)$. For all the generators but $i_1(0, 1)$  and $i_2((0, 1))$ it is straightforward. For $i_1((0, 1))$ we have
\begin{multline*}\psi\circ\phi\left(i_1(0, 1)\right) = \psi(\phi_1((0, 1))) = \psi(\sum_{\beta} p_{\beta}) = i_2((\sum_{\beta} \bar{\bar p}_{\beta}, 0))\\
= i_2\circ\theta_2(\delta_{n+2}) = i_1\circ\theta_1(\delta_{n+2}) = i_1\left((0, 1)\right). \end{multline*}
For  $i_2((0, 1))$ computations are analogous.
\end{proof}

\begin{remark} For graphs with infinitely many vertices, one can obtain a version of Theorem \ref{case1} using non-unital free products and amalgamation over $c_0$ instead of $\mathbb C^{n+1}$. For  Theorem \ref{case2} the assumption that $G$ has  finitely many vertices, or in other words, $C^*(G)$ is unital, seems to be essential. In $C^*(G)$ all the generating projections are orthogonal. When neither of $G_1$ and $G_2$ contains all the vertices, in the amalgamated free product the images of $p_{\alpha}$'s would not be orthogonal to the images of $p_{\beta}$'s unless all the maps are unital which would guarantee that all the projections would sum up to 1 and therefore be orthogonal.
\end{remark}

\section{Graphs with RFD C*-algebras}


 The following lemma is easy. In fact a more general statement -- where the class of RFD C*-algebras will be replaced by a larger class of finite C*-algebras -- will be proved in Lemma \ref{homToFinite}.

\begin{lemma}\label{necessaryCondition} Let $G$ be  a graph. If $C^*(G)$ is RFD, then no edge of $G$ enters a cycle.
\end{lemma}

\begin{lemma}\label{disjointCycles} Let $G$ be  a graph. If no edge of $G$ enters a cycle, then all cycles in $G$ are disjoint.
\end{lemma}
\begin{proof} Let $\mu_1, \mu_2$ be two cycles and suppose they are not disjoint. Let $p$ be a common vertex of them. There is an edge $e_1^{(2)}\in \mu_2$ with range in $p$. Since no edge of $G$ enters a cycle, $e_1^{(2)}\in \mu_1$. Let $p_1$ be the source of $e_1^{(2)}$.  There is an edge $e_2^{(2)}\in \mu_2$ with range in $p_1$. Therefore $e_2^{(2)}\in \mu_1$. Continuing this process, we obtain that all edges of $\mu_2$ are in $\mu_1$.
\end{proof}

\begin{lemma}\label{RFDfinManyIdeals} If an RFD C*-algebra has finitely many ideals, then it is finite-dimensional. 
\end{lemma}
\begin{proof} Let $A$ be such a C*-algebra. Since it is RFD, there is a separating family $\{\pi_n, n\in \mathbb N\}$ of irreducible finite-dimensional representations of $A$. Since $A$ has finitely many ideals, the set $\{ker\;\pi_n, n\in \mathbb N\}$ is finite. Since two finite-dimensional irreducible representations have the same kernel if and only of they are unitarily equivalent, there are finitely many $\pi_{i_1}, \ldots, \pi_{i_N}$ that separate points of $A$. Then $A$ is isomorphic to the finite-dimensional C*-algebra $(\oplus_{k=1}^N \pi_{i_k})(A)$. 
\end{proof}

\begin{theorem}\label{RFD} Let $G$ be a graph with finitely many vertices. Then $C^*(G)$ is RFD if and only if $G$ is finite and no cycle has an entry.
\end{theorem}
\begin{proof}
First, assume that no cycle in $G$ has an entry. By Lemma \ref{disjointCycles} all cycles in $G$ are disjoint. Let $C= \{\mu_1, \ldots, \mu_L\}$ be the set of all cycles in $G$. Let
$$G_1 = \bigcup_{\mu\in C} \mu, $$ and let $G_2$ be the rest of the graph, that is, it consists of edges  that are not in $G_1$ and their vertices.  Then $G_2$ is necessarily a "forest", that is, it has no cycles. The condition that no cycle has an entry implies that no edge of $G_2$ enters $G_1$.
By Theorem \ref{case1} in the case when $G^0_2 = G^0$ (and by Theorem \ref{case2} in the case $G^0_2 \neq G^0$, respectively), we obtain

\bigskip

 {\it $C^*(G)$ is an amalgamated free product of $C^*(G_1)\oplus \mathbb C$ and $C^*(G_2)$ ($C^*(G_2)\oplus \mathbb C$, respectively), where the amalgamation is done over a finite-dimensional C*-subalgebra.}

\bigskip

Now suppose $C^*(G)$ is RFD. By Lemma \ref{necessaryCondition} no  cycle has an entry. Then $C^*(G)$ is amalgamated free product as above. Then $C^*(G_2)$, being a C*-subalgebra of $C^*(G)$, is RFD. Since $G_2$ is a forest, $C^*(G_2)$ is an AF-algebra (\cite[Th. 2.7 + Cor. 5.3]{Raeburn} or \cite[Rem. 2.4.14]{Tomforde}), and by \cite{Eilers} it has finitely many ideals. By Lemma \ref{RFDfinManyIdeals} $C^*(G_2)$ is  finite-dimensional. Therefore $G_2$ is finite.   Since cycles in $G$ are disjoint by Lemma \ref{disjointCycles}, and there are finitely many vertices in $G$, $G_1$ is also a finite graph. So $G$ is finite.

Now we assume that $G$ is finite and no cycle has an entry. We want to prove that $C^*(G)$ is RFD.
We are going to construct an embedding $\pi_1: C^*(G_1)\oplus \mathbb C$ into product of some matrix algebras and embedding $\pi_2$ of $C^*(G)$ ($C^*(G_2)\oplus \mathbb C$, respectively) into product of the same matrix algebras such that $\pi_1\circ \theta_1 = \pi_2\circ\theta_2$, where $\theta_1, \theta_2$ are as in Theorem \ref{case1} (Theorem \ref{case2}, respectively).  Then Theorem \ref{LiShen} will finish the proof.

\medskip

Our constructions will be similar in both cases, so we carry the two
cases along, underlying the differences when needed.
 Let $I_{\mu}$ be the number of vertices shared by the cycle $\mu$ and $G_2$. Let $I = \sum_{\mu\in C}I_{\mu}$ be the number of common vertices of $G_1$ and $G_2$.  For a source $t\in G_2$ let $n(t) = \sharp\{\nu\in G_2^*\;|\; s(\nu) = t\}.$ Let $$k = \sum_{t: t \; \text{is a  source in}\; G} n(t). $$
 It follows from Theorem \ref{tree} that there is an embedding
  $$\pi: C^*(G_2) \to M_k\cong B(\mathbb C^{k})$$ such that $\pi(t)$ is a rank 1 projection, for each source $t$ in $G_2$. Note that since no edge of $G_2$ enters $G_1$, the common vertices of $G_1$ and $G_2$ have to be sources of $G_2$.  Therefore $\pi$ sends them to rank 1 projections and WLOG we can assume that they are sent to the projections on the last $I$ basis vectors so that the $I_1$ common vertices of $\mu_1$ and $G_2$  go to the projections on the first $I_1$ of those $I$ basis vectors, the $I_2$ common vertices of $\mu_2$ and $G_2$  go to the projections on the next $I_2$ of those $I$ basis vectors, and so on.

 We have $C^*(G_1) \cong \oplus_{\mu} C^*(\mu).$ For a cycle $\mu$, let $N_{\mu}$ be the number of vertices in $\mu$.
  For a cycle $\mu$ and $z\in \mathbb T$, let $$\rho_{\mu, z}: C^*(\mu) \to M_{N_{\mu}}\cong B(\mathbb C^{N_{\mu}})$$ be as in Corollary \ref{cycle}. WLOG we can assume that the $I_{\mu}$  common vertices of $\mu$ and $G_2$ are sent to the projections on the  first $I_{\mu}$  basis vectors of $\mathbb C^{N_{\mu}}$.

 Now we are going to construct representations $\pi_{1, z}$ of $C^*(G_1) \oplus \mathbb C$ and  $\tilde \pi_{2}$ of $C^*(G_2)$ ($C^*(G_2)\oplus \mathbb C$, in the second case respectively)  on the space of dimension  $(k+\sum_{\mu}N_{\mu} - I)$. (We note that in the first case $k+\sum_{\mu}N_{\mu} - I = k$).
 We will view $M_k$ and $ M_{\sum_{\mu}N_{\mu}}$ as C*-subalgebras of $M_{k+\sum_{\mu}N_{\mu} - I}$ via the embeddings $$M_k\hookrightarrow M_{k+\sum_{\mu}N_{\mu} - I}, \; \;\; T\mapsto \left(\begin{array}{cc} T & \\ & 0_{\sum N_{\mu} - I}\end{array}\right),$$ and
  $$M_{\sum_{\mu}N_{\mu}}\hookrightarrow M_{k+\sum_{\mu}N_{\mu} - I}, \; \; \; T \mapsto \left(\begin{array}{cc} 0_{k-I} & \\& T\end{array}\right).$$ For each $\mu_l$ we embed
  $$M_{N_{\mu_l}}\cong B(\mathbb C^{I_{\mu_l}} \oplus \mathbb C^{N_{\mu_l} - I_{\mu_l}})\hookrightarrow
 M_{\sum_{\mu}N_{\mu}}\cong B(\mathbb C^{I_{\mu_1}} \oplus \ldots \oplus \mathbb C^{I_{\mu_L}} \oplus \mathbb C^{N_{\mu_1} - I_{\mu_1}}\oplus \ldots \oplus \mathbb C^{N_{\mu_L} - I_{\mu_L}})$$ via the canonical embedding
 $$\mathbb C^{I_{\mu_l}} \oplus \mathbb C^{N_{\mu_l} - I_{\mu_l}}\hookrightarrow
 \mathbb C^{I_{\mu_1}} \oplus \ldots \oplus \mathbb C^{I_{\mu_L}} \oplus \mathbb C^{N_{\mu_1} - I_{\mu_1}}\oplus \ldots \oplus \mathbb C^{N_{\mu_L} - I_{\mu_L}}.$$  These embeddings are summarized on Figure \ref{embeddings}. 

\begin{figure}
 \begin{center}
\begin{tikzpicture}[scale=1]

\def\k{3}           
\def\I{2}         
\def\N{4}         
\def\B{\N}          
\def\S{5}         

\coordinate (A1) at (0,\S);
\coordinate (A2) at (\k,\S-\k);
\draw[thick] (A1) rectangle (A2);

\coordinate (B1) at (\k-\I,\S-\k-\B+\I);
\coordinate (B2) at (\k-\I+\B,\S-\k-\B+\I+\B);
\draw[thick] (B1) rectangle (B2);

\draw[dashed] (\k,\S) -- (\S,\S);
\draw[dashed] (\S,\S) -- (\S,\N);

\draw[dashed] (0,\S-\k) -- (0,0);
\draw[dashed] (0,0) -- (\S-\N,0);


\draw[decorate,decoration={brace, amplitude=5pt}, thick]
(0,\S+0.2) -- (\k,\S+0.2)
node[midway, above=6pt] {\footnotesize $k$};

\draw[decorate,decoration={brace, amplitude=5pt}, thick]
(\k-\I,\S-\k+0.2) -- (\k,\S-\k+0.2)
node[midway, above=6pt] {\footnotesize $I$};

\draw[decorate,decoration={brace, amplitude=5pt}, thick]
(\S-\N,+0.2) -- (\S,+0.2)
node[midway, above=6pt] {\footnotesize $\sum N_i $};

\end{tikzpicture}
\caption{Embeddings}
\label{embeddings}
\end{center}
\end{figure}

 The constructions of $\pi$ and $\rho_{\mu, z}$ imply that

 \begin{equation}\label{RFD1} \rho_{\mu, z}(\bar p_i) = \pi(\bar{\bar p}_i), \;\text{when}\; p_i \;\text {is a common vertex of } \; \mu \;\text{and}\; G_2,\end{equation}

 \begin{equation}\label{RFD2} \rho_{\mu, z}\left(\sum_{\alpha: p_{\alpha} \in \mu} \bar p_{\alpha}\right) = \mathbb{1}_{N_{\mu}- I_{\mu}}, \end{equation}

 \begin{equation}\label{RFD3} \pi\left(\sum_{\beta} \bar{\bar p}_{\beta}\right) = \mathbb {1}_{k-I}.
 \end{equation}

 Now we define  a representation $\pi_{1, z}: C^*(G_1)\oplus \mathbb C \to M_{k+\sum_{\mu}N_{\mu} - I}$ by

 $$\pi_{1, z}\left((\oplus_{\mu} a_{\mu}, \lambda)\right) = \left(\oplus_{\mu} \rho_{\mu, z}(a_{\mu})\right) \oplus \lambda\mathbb{1}_{k-I},$$

 \noindent for any $a_{\mu} \in C^*(\mu), \mu \in C, \lambda\in \mathbb C$.
 We define a representation  $$\tilde\pi_{2}: C^*(G_2) \to M_{k+\sum_{\mu}N_{\mu} - I},$$ $$(\tilde \pi_{2}: C^*(G_2)\oplus \mathbb C \to M_{k+\sum_{\mu}N_{\mu} - I}, \;\text{respectively})$$ by
  $$\tilde\pi_2(a) = \pi(a)$$
  $$(\tilde\pi_2((a, \lambda)) = \pi(a) \oplus \lambda \mathbb{1}_{\sum_{\mu} N_{\mu} - I}, \;\text{respectively}),$$
 for any $a\in C^*(G_2)$ (and $\lambda\in \mathbb C$, respectively). Let us show that $\pi_{1, z}\circ \theta_1 = \tilde \pi_2\circ \theta_2$.
 In the case $G^0_2 = G^0$
 \begin{multline*} \pi_{1, z} \circ \theta_1\left((\lambda_1, \ldots, \lambda_{n+1})\right)) = \pi_{1, z} \left((\sum_i\lambda_i \bar  p_i, \lambda_{n+1})\right)\\
 = \left(\oplus_{\mu} \rho_{\mu, z}(\sum_{i: p_i\in \mu} \lambda_i\bar p_i)\right) \oplus \lambda_{n+1}\mathbb{1}_{k-I},
 \end{multline*}
 \begin{multline*} \tilde \pi_2\circ \theta_2 \left((\lambda_1, \ldots, \lambda_{n+1})\right) = \tilde \pi_2\left(\sum_i \lambda_i \bar{\bar p}_i + \lambda_{n+1} \sum_{\beta} \bar{\bar p}_{\beta}\right) \\= \pi\left(\sum_i \lambda_i \bar{\bar p}_i + \lambda_{n+1}\sum_{\beta}\bar{\bar p}_{\beta}\right)
 \stackrel{(\ref{RFD3})}{=} \sum_{\mu} \sum_{i: p_i\in \mu} \lambda_i \pi(\bar{\bar p}_i) \oplus \lambda_{n+1}\mathbb{1}_{k-I}.\nonumberthis\end{multline*}
 Using (\ref{RFD1}) we conclude that   $\pi_{1, z}\circ \theta_1 = \tilde \pi_2\circ \theta_2$ in this case.

 \noindent In the case $G^0_2 \neq G^0$

 \begin{multline*} \pi_{1, z}\circ\theta_1\left((\lambda_1, \ldots, \lambda_{n+2})\right) = \pi_{1,z}\left((\sum_i \lambda_i\bar p_i + \lambda_{n+1}\sum_{\alpha} \bar p_{\alpha}, \lambda_{n+2})\right) \\ = \pi_{1,z}\left((\sum_{\mu}\sum_{i: p_i\in \mu} \lambda_i\bar p_i + \lambda_{n+1}\sum_{\mu}\sum_{\alpha: p_{\alpha}\in \mu} \bar p_{\alpha}, \lambda_{n+2})\right)\\ \stackrel{(\ref{RFD2})}{=}
 \sum_{\mu}\sum_{i: p_i\in \mu} \lambda_i\rho_{\mu, z}(\bar p_i) \oplus \lambda_{n+1} \sum_{\mu} \mathbb{1}_{N_{\mu}-I_{\mu}} \oplus \lambda_{n+2}\mathbb{1}_{k-I}\\= \sum_{\mu}\sum_{i: p_i\in \mu} \lambda_i\rho_{\mu, z}(\bar p_i) \oplus \lambda_{n+1} \mathbb{1}_{\sum_{\mu}N_{\mu} - I} \oplus\lambda_{n+2}\mathbb{1}_{k-I}, \nonumberthis
 \end{multline*}
 \begin{multline*} \tilde\pi_2\circ \theta_2\left((\lambda_1, \ldots, \lambda_{n+2})\right) = \tilde \pi_2\left((\sum_i \lambda_i \bar{\bar p}_i +
 \lambda_{n+2}\sum_{\beta}\bar{\bar p}_{\beta}, \lambda_{n+1})\right)\\
 = \sum_i \lambda_i \pi(\bar{\bar p}_i)\oplus \lambda_{n+2}\sum_{\beta}\pi(\bar{\bar p}_{\beta})\oplus \lambda_{n+1}\mathbb{1}_{\sum_{\mu}N_{\mu}-I}
 \\ \stackrel{(\ref{RFD3})}{=} \sum_{\mu}\sum_{i: p_i\in \mu} \lambda_i \pi(\bar{\bar p}_i) \oplus \lambda_{n+2}\mathbb{1}_{k-I} \oplus \lambda_{n+1}\mathbb{1}_{\sum_{\mu} N_{\mu} - I}.\nonumberthis
 \end{multline*}
 By (\ref{RFD1}), we see that   $\pi_{1, z}\circ \theta_1 = \tilde \pi_2\circ \theta_2$ in this case either. So, for each $z\in \mathbb T$
 \begin{equation}\label{RFD4} \pi_{1, z}\circ \theta_1 = \tilde \pi_2\circ \theta_2. \end{equation}

 Let $\mathcal F$ be a dense subset of $\mathbb T$.  Let

 $$\pi_1 =  \bigoplus_{z\in \mathcal F} \pi_{1, z},$$

 $$\pi_2 = \bigoplus_{z\in \mathcal F} \tilde \pi_2.$$

 Then $\pi_1, \pi_2$ are unital $\ast$-homomorphisms, and (\ref{RFD4}) implies that
 $$\pi_1\circ\theta_1 = \pi_2\circ \theta_2.$$
  For each cycle $\mu$,  $\bigoplus_{z\in \mathcal F} \rho_{\mu, z}$ is injective by Corollary \ref{cycle}. Therefore so is $\pi_1$.
  Since by Theorem \ref{tree} $\pi$ is is injective, so is $\tilde \pi_2$ and therefore so is  $\pi_2$.
  By Theorem \ref{LiShen}, $C^*(G)$ is RFD.
\end{proof}

\section{Operator norm stability of unital graph C*-algebras}

Here we consider only unital graph C*-algebras, that is C*-algebras of graphs that have finitely many vertices and possibly infinitely many edges.

From now on we always will denote vertices of graphs by the letters $p, q$ and edges of graphs by the letters $e, f$. By this reason we do not have to write  $p\in G^0$ and $e\in G^1$ anymore, instead will write $p, f\in G$.

\medskip

We say that {\bf a path $\nu $ leads to a cycle $c$ } if $\nu$ ends at a vertex of $c$ and no edge of $\nu$ belongs to $c$.

\medskip

Let $G' $ be the subgraph of $G$ that consists of all paths that lead to cycles. Precisely
$$ e\in G'^{(1)} \;\text{if}\; \exists \; \nu\in G^* \;\text{such that}\;  e\in \nu, \;\nu \;\text{leads to a cycle}.$$

\medskip

\tikzset{every loop/.style={min distance=10mm,in=120,out=240,looseness=40}}
\begin{figure}
\centering
 \begin{tikzpicture}[node distance={15mm}, thick, main/.style = {draw, fill, circle, inner sep=1.5pt},mainLarge/.style = {draw, circle, minimum size=30mm}]

 \def\Unit{1.15}

\node[main, label=below:$ $,red] (P1) at (0.5*\Unit,0*\Unit){};
\node[main, label=below:$ $,red] (P2) at (1.5*\Unit,-1*\Unit){};
\node[main, label=below:$ $,black] (P3) at (2.5*\Unit,0*\Unit){};
\node[main, label=below:$ $,black] (P4) at (2*\Unit,1*\Unit){};
\node[main, label=below:$ $,black] (P5) at (1*\Unit,1*\Unit){};
\node[main, label=below:$ $,red] (P6) at (3*\Unit,0*\Unit){};
\node[main, label=below:$ $,black] (P7) at (4*\Unit,0*\Unit){};
\node[main, label=below:$ $,black] (P7b) at (5*\Unit,0*\Unit){};
\node[main, label=below:$ $,red] (P8) at (3*\Unit,1*\Unit){};
\node[main, label=below:$ $,red] (P9) at (1*\Unit,2*\Unit){};
\node[main, label=below:$ $,red] (P10) at (2*\Unit,2*\Unit){};
\node[main, label=below:$ $,red] (P11) at (3*\Unit,2*\Unit){};
\node[main, label=below:$ $,black] (P12) at (4*\Unit,2*\Unit){};
\node[main, label=below:$ $,black] (P13) at (5*\Unit,2*\Unit){};
\node[main, label=below:$ $,black] (P14) at (6*\Unit,2*\Unit){};
\node[main, label=below:$ $,black] (P15) at (7*\Unit,2*\Unit){};
\node[main, label=below:$ $,black] (P16) at (2*\Unit,3*\Unit){};
\node[main, label=below:$ $,red] (P17) at (3*\Unit,3*\Unit){};
\node[main, label=below:$ $,black] (P18) at (4*\Unit,3*\Unit){};
\node[main, label=below:$ $,black] (P19) at (5*\Unit,3*\Unit){};
\node[main, label=below:$ $,black] (P20) at (6*\Unit,3*\Unit){};
\node[main, label=below:$ $,red] (P21) at (3*\Unit,4*\Unit){};

\path[-angle 90,font=\scriptsize]

(P1) edge ["",red]   (P2)
(P1) edge ["",bend left=20,red]   (P2)
(P2) edge ["",black]   (P3)
(P3) edge ["",black]   (P4)
(P4) edge ["",black]   (P5)
(P5) edge ["",black]   (P1)

(P8) edge ["",red]   (P4)
(P8) edge ["",red]   (P6)

(P6) edge ["",black]   (P7)
(P6) edge [loop left,out=225, in=315,  looseness=20,"",black]   (P6)
(P7b) edge ["",black]   (P7)
(P7) edge ["",black]   (P12)

(P9) edge[out=135, in=225, looseness=20, red] (P9)  
(P9) edge[out=45,  in=135,  looseness=20, red] (P9)  
(P9) edge[out=225, in=315, looseness=20, red] (P9)  


(P10) edge ["",red]   (P9)
(P10) edge ["",black]   (P16)

(P11) edge ["",red]   (P10)
(P11) edge ["",black]   (P12)
(P11) edge ["",red,double, double distance=2pt,-{Stealth[length=10pt]}]   (P8)

(P12) edge ["",black,double, double distance=2pt,-{Stealth[length=10pt]}]   (P13)

(P13) edge ["",black]   (P14)

(P14) edge ["",black]   (P15)

(P17) edge ["",red]   (P11)
(P17) edge ["",black]   (P18)
(P17) edge [loop left, looseness=20,out=135, in=225,"",red]   (P17)

(P18) edge ["",black,double, double distance=2pt,-{Stealth[length=10pt]}]   (P19)

(P19) edge ["",black]   (P20)

(P21) edge ["",red]   (P17)
;

\end{tikzpicture}

   \caption{Example of $G$ and $G'$.}
\label{fig1}
\end{figure}

{\bf Example } Figure \ref{fig1} contains an example of a graph. $G'$  is coloured in red. By $\Rightarrow$ we draw "edges of infinite multiplicity", that is infinitely many edges with the same source and the same range.

\medskip
\newpage

Let $S\subset G$ and $e\in G^1$. We will say that $e$ {\bf can be reached from} $S$ if there is a path $\nu = (e, \nu_n, \ldots, \nu_1)$ with $s(\nu)\in S$.
(Here we allow the case $\nu=e$).

\medskip

We define a subgraph $\tilde G$ of $G$. First, we specify what are edges of $\tilde G$.

\medskip

An edge $e$ belongs to $\tilde G$ if one of the following four conditions holds : 

\begin{figure}[htbp]
    \centering
    \begin{minipage}{0.55\textwidth}
    (1) $e$ can be reached from a vertex of   a cycle  all of whose edges are not in $G'$.

\vspace{1cm}
    On the picture on the right, an example of $e$ satisfying (1) is shown in blue. $G'$ is shown in red.
    \end{minipage}
    \hfill
    \begin{minipage}{0.35\textwidth}
        \centering
        \begin{tikzpicture}[node distance={15mm}, thick, main/.style = {draw, fill, circle, inner sep=1.5pt},mainLarge/.style = {draw, circle, minimum size=30mm}]

 \def\Unit{1}

\node[main, label=below:$ $,red] (P1) at (0*\Unit,0*\Unit){};
\node[main, label=below:$ $,red] (P2) at (0*\Unit,1*\Unit){};
\node[main, label=below:$ $,red] (P3) at (0*\Unit,2*\Unit){};
\node[main, label=below:$ $,red] (P4) at (0*\Unit,3*\Unit){};
\node[main, label=below:$ $,blue] (P5) at (1*\Unit,0.5*\Unit){};
\node[main, label=below:$ $,blue] (P6) at (2*\Unit,1*\Unit){};

\path[-angle 90,font=\scriptsize]
(P1) edge [loop left,"",black]   (P1)

(P1) edge ["",black]   (P5)
(P5) edge ["e",blue]   (P6)

(P4) edge ["",red]   (P3)
(P3) edge ["",red]   (P2)
(P2) edge ["",red]   (P1)
;

\end{tikzpicture}
\label{fig1a}
    \end{minipage}
\end{figure}


\begin{figure}[htbp]
    \begin{minipage}{0.55\textwidth}
    (2) $e$ cannot be reached  from  $G'$.

    \vspace{1cm}
    On the picture on the right, an example of $e$ satisfying (2) is shown in blue. $G'$ is shown in red.
    \end{minipage}
    \hfill
    \begin{minipage}{0.35\textwidth}
        \centering
        \begin{tikzpicture}[node distance={15mm}, thick, main/.style = {draw, fill, circle, inner sep=1.5pt},mainLarge/.style = {draw, circle, minimum size=30mm}]

 \def\Unit{1}

\node[main, label=below:$ $,red] (P1) at (0*\Unit,0*\Unit){};
\node[main, label=below:$ $,red] (P2) at (0*\Unit,1*\Unit){};
\node[main, label=below:$ $,red] (P3) at (0*\Unit,2*\Unit){};

\node[main, label=below:$ $,blue] (P4) at (1*\Unit,0.5*\Unit){};
\node[main, label=below:$ $,blue] (P5) at (2*\Unit,0*\Unit){};

\path[-angle 90,font=\scriptsize]
(P1) edge [loop left,"",black]   (P1)

(P1) edge ["",black]   (P4)
(P5) edge ["e",blue]   (P4)

(P3) edge ["",red]   (P2)
(P2) edge ["",red]   (P1)
;

\end{tikzpicture}
\label{fig2a}
    \end{minipage}
    \end{figure}
\begin{figure}[htbp]
    \begin{minipage}{0.55\textwidth}
    (3) $e$ can be reached from a vertex of an edge as in 2).

    \vspace{1cm}
    On the picture on the right, an example of $e$ satisfying (3) is shown in blue. $G'$ is shown in red.
    \end{minipage}
    \hfill
    \begin{minipage}{0.35\textwidth}
        \centering
        \begin{tikzpicture}[node distance={15mm}, thick, main/.style = {draw, fill, circle, inner sep=1.5pt},mainLarge/.style = {draw, circle, minimum size=30mm}]

 \def\Unit{1}

\node[main, label=below:$ $,red] (P1) at (0*\Unit,0*\Unit){};
\node[main, label=below:$ $,red] (P2) at (0*\Unit,1*\Unit){};
\node[main, label=below:$ $,red] (P3) at (0*\Unit,2*\Unit){};

\node[main, label=below:$ $,black] (P4) at (1*\Unit,0.5*\Unit){};
\node[main, label=below:$ $,black] (P5) at (2*\Unit,0*\Unit){};

\node[main, label=below:$ $,black] (P6) at (1*\Unit,1*\Unit){};
\node[main, label=below:$ $,blue] (P7) at (2*\Unit,1*\Unit){};
\node[main, label=below:$ $,blue] (P8) at (3*\Unit,1*\Unit){};

\path[-angle 90,font=\scriptsize]
(P1) edge [loop left,"",black]   (P1)

(P1) edge ["",black]   (P4)
(P5) edge ["",black]   (P4)
(P5) edge ["",black]   (P7)

(P3) edge ["",red]   (P2)
(P2) edge ["",red]   (P1)

(P2) edge ["",black]   (P6)
(P6) edge ["",black]   (P7)
(P7) edge ["e",blue]   (P8)
;

\end{tikzpicture}
\label{fig3a}
    \end{minipage}\vspace{1cm} 
    \\
    \begin{minipage}{0.55\textwidth}
    (4) there are vertices $p, q$ and infinitely many edges $f_i$,  such that $s(f_i)=p$, $r(f_i)=q$, $f_i\notin G'$, and $e$ can be reached from $q$

    \vspace{1cm}
    On the picture on the right, an example of $e$ satisfying (4) is shown in blue. $G'$ is shown in red.
    \end{minipage}
    \hfill
    \begin{minipage}{0.35\textwidth}
        \centering
         \begin{tikzpicture}[node distance={15mm}, thick, main/.style = {draw, fill, circle, inner sep=1.5pt},mainLarge/.style = {draw, circle, minimum size=30mm}]

 \def\Unit{1}

\node[main, label=below:$ $,red] (P1) at (0*\Unit,0*\Unit){};
\node[main, label=below:$ $,red] (P2) at (0*\Unit,1*\Unit){};
\node[main, label=below:$ $,red] (P3) at (0*\Unit,2*\Unit){};
\node[main, label=below:$ $,blue] (P4) at (1*\Unit,1*\Unit){};
\node[main, label=below:$ $,blue] (P5) at (2*\Unit,1*\Unit){};

\path[-angle 90,font=\scriptsize]
(P1) edge [loop left,"",black]   (P1)

(P2) edge ["",black,double, double distance=2pt,-{Stealth[length=10pt]}]   (P4)

(P4) edge ["e",blue]   (P5)

(P3) edge ["",red]   (P2)
(P2) edge ["",red]   (P1)
;

\end{tikzpicture}
\label{fig4a}
    \end{minipage}
\end{figure}

\FloatBarrier

\medskip

 Now we specify what are vertices of $\tilde G$: a vertex $p$ belongs to $\tilde G$ if either $p$ is a vertex of some edge $e\in \tilde G$, or $p$ is an infinite receiver and all the edges that it receives are not in $G'$, or $p$ is an isolated vertex of $G$ (that is, it neither receives nor emits any edges).

\medskip

Some examples are given in Figure \ref{ExempleGtilde}.



\tikzset{every loop/.style={min distance=10mm,in=120,out=240,looseness=40}}
\begin{figure}
\centering
\begin{subfigure}{0.8\textwidth}
\centering
 \begin{tikzpicture}[node distance={15mm}, thick, main/.style = {draw, fill, circle, inner sep=1.5pt},mainLarge/.style = {draw, circle, minimum size=30mm}]

 \def\Unit{1.5}

\node[main, label=below:$ $,black] (P1) at (0.5*\Unit,0*\Unit){};
\node[main, label=below:$ $,black] (P2) at (1.5*\Unit,-0.75*\Unit){};
\node[main, label=below:$ $,black] (P3) at (2.5*\Unit,0*\Unit){};
\node[main, label=below:$ $,black] (P4) at (2*\Unit,1*\Unit){};
\node[main, label=below:$ $,black] (P5) at (1*\Unit,1*\Unit){};
\node[main, label=below:$ $,black] (P6) at (3*\Unit,0*\Unit){};
\node[main, label=below:$ $,black] (P7) at (4*\Unit,0*\Unit){};
\node[main, label=below:$ $,black] (P7b) at (5*\Unit,0*\Unit){};
\node[main, label=below:$ $,black] (P8) at (3*\Unit,1*\Unit){};
\node[main, label=below:$ $,black] (P9) at (1*\Unit,2*\Unit){};
\node[main, label=below:$ $,black] (P10) at (2*\Unit,2*\Unit){};
\node[main, label=below:$ $,black] (P11) at (3*\Unit,2*\Unit){};
\node[main, label=below:$ $,black] (P12) at (4*\Unit,2*\Unit){};
\node[main, label=below:$ $,black] (P13) at (5*\Unit,2*\Unit){};
\node[main, label=below:$ $,black] (P14) at (6*\Unit,2*\Unit){};
\node[main, label=below:$ $,black] (P15) at (7*\Unit,2*\Unit){};
\node[main, label=below:$ $,black] (P16) at (2*\Unit,3*\Unit){};
\node[main, label=below:$ $,black] (P17) at (3*\Unit,3*\Unit){};
\node[main, label=below:$ $,black] (P18) at (4*\Unit,3*\Unit){};
\node[main, label=below:$ $,black] (P19) at (5*\Unit,3*\Unit){};
\node[main, label=below:$ $,black] (P20) at (6*\Unit,3*\Unit){};
\node[main, label=below:$ $,black] (P21) at (3*\Unit,4*\Unit){};

\path[-angle 90,font=\scriptsize]

(P1) edge ["",red]   (P2)
(P1) edge ["",bend left=20,red]   (P2)
(P2) edge ["",black]   (P3)
(P3) edge ["",black]   (P4)
(P4) edge ["",black]   (P5)
(P5) edge ["",black]   (P1)

(P8) edge ["",red]   (P4)
(P8) edge ["",red]   (P6)

(P6) edge ["",blue]   (P7)
(P6) edge [loop left,out=225, in=315,  looseness=20,"",blue]   (P6)
(P7b) edge ["",blue]   (P7)
(P7) edge ["",blue]   (P12)

(P9) edge[out=135, in=225, looseness=20, red] (P9)  
(P9) edge[out=45,  in=135,  looseness=20, red] (P9)  
(P9) edge[out=225, in=315, looseness=20, red] (P9)  


(P10) edge ["",red]   (P9)
(P10) edge ["",black]   (P16)

(P11) edge ["",red]   (P10)
(P11) edge ["",black]   (P12)
(P11) edge ["",red,double, double distance=2pt,-{Stealth[length=10pt]}]   (P8)

(P12) edge ["",blue,double, double distance=2pt,-{Stealth[length=10pt]}]   (P13)

(P13) edge ["",blue]   (P14)

(P14) edge ["",blue]   (P15)

(P17) edge ["",red]   (P11)
(P17) edge ["",black]   (P18)
(P17) edge [loop left, looseness=20,out=135, in=225,"",red]   (P17)

(P18) edge ["",black,double, double distance=2pt,-{Stealth[length=10pt]}]   (P19)

(P19) edge ["",blue]   (P20)

(P21) edge ["",red]   (P17)
;

\end{tikzpicture}
   \caption*{Example a)}
\end{subfigure}
%
\par

\begin{subfigure}{0.8\textwidth}
\centering

\begin{tikzpicture}[node distance={15mm}, thick, main/.style = {draw, fill, circle, inner sep=1.5pt},mainLarge/.style = {draw, circle, minimum size=30mm}]

 \def\Unit{1}

\node[main, label=below:$ $,black] (P1) at (0*\Unit,0*\Unit){};
\node[main, label=below:$ $,black] (P2) at (0*\Unit,1*\Unit){};
\node[main, label=below:$ $,black] (P3) at (0*\Unit,2*\Unit){};
\node[main, label=below:$ $,black] (P4) at (0*\Unit,3*\Unit){};
\node[main, label=below:$ $,blue] (P5) at (1*\Unit,2*\Unit){};

\path[-angle 90,font=\scriptsize]
(P1) edge [loop left,"",blue]   (P1)

(P2) edge ["",red]   (P1)
(P3) edge ["",red]   (P2)
(P4) edge ["",red]   (P3)

(P3) edge ["",black,double, double distance=2pt,-{Stealth[length=10pt]}]   (P5)
;

\end{tikzpicture}
 \caption*{Example b)}
\end{subfigure}
\caption{Examples of $G$ with $G'$ in red and $\tilde{G}$ in blue}
\label{ExempleGtilde}
\end{figure}

\medskip

We are going to prove that the operator norm stability of $C^*(G)$ is completely determined by $\tilde G$.

\medskip


\begin{lemma}\label{notInTildeG} If $e\in G'$, then $e\notin \tilde G$.
\end{lemma}
\begin{proof} Follows immediately from the definition of $\tilde G$.
\end{proof}

\begin{lemma}\label{lemma2MSP} 1) If $e_1\in \tilde G$ and  $s(e)=r(e_1)$, then $e\in \tilde G$,

2) If $e_1\in \tilde G$ and  $s(e)=s(e_1)$, then $e\in \tilde G$.
\end{lemma}
\begin{proof} 1) If $e_1$ satisfies the condition (1) from the definition of $\tilde G$, then $e$ also satisfies the condition (1).
If $e_1$ satisfies the condition (2) or (3), then $e$ satisfies the condition (3). If $e_1$ satisfies the condition (4), then $e$ satisfies the condition (4).

2) Since the four conditions in the definition of $\tilde G$ are formulated in terms of reaching $e$ from some subset of graph, and $s(e)=s(e_1)$, each of these conditions either hold or does not hold  for both $e$ and $e_1$ simultaneously.
\end{proof}

Recall that a C*-algebra is called {\it finite} if no projection in $A$ is equivalent to its proper subprojection.

\begin{lemma}\label{homToFinite} Let $G$ be a graph, $A$ a finite C*-algebra, and $\pi: C^*(G)\to A$ a $\ast$-homomorphism. Let $f\in G$ be an edge that enters a cycle. Then $\pi(f)=0$.
\end{lemma}
\begin{proof} Suppose $\pi(f)\neq 0$.

\noindent We can assume that $f$ enters a cycle $(e_1,\ldots, e_n)$ at the vertex $p_1$. Since for $i=1, \ldots, n-1$
$$e_i^*e_i = p_i, \;\text{and} \; e_ie_i^* \le p_{i+1},$$ we conclude that $p_i$ is equivalent to a subprojection of $p_{i+1}$, and consequently
\begin{equation}\label{homToFinite1} \pi(p_i)\; \text{ is equivalent to a subprojection of}\; \pi(p_{i+1}), \end{equation}
$i=1, \ldots, n-1$. Since $$e_n^*e_n= p_n, \;\text{and}\;  e_ne_n^* + ff^* \le  p_1,$$
(the latter one is because $e_n$ and $f$ have mutually orthogonal ranges), and $\pi(f)\neq 0$, we conclude that
\begin{equation} \label{homToFinite2} \pi(p_n)\; \text{ is equivalent to a proper subprojection of}\; \pi(p_{1}). \end{equation}
From (\ref{homToFinite1}) and (\ref{homToFinite2}) it follows that $\pi(p_1)$ is equivalent to a proper subprojection of itself. Since $A$ is finite, we conclude that $\pi(p_1)=0$. But then, since $\pi(f)\pi(f)^* \le \pi(p_1)$, $\pi(f)=0$, contradiction.

\end{proof}

\begin{lemma}\label{lemma3MSP} Let $A$ be a finite C*-algebra and $\pi: C^*(G)\to A$ a $\ast$-homomorphism. Let  $e$ be an edge such that  $s(e) = s(f)$, for some $f\in G'$. Then $\pi(e)=0$.
\end{lemma}
\begin{proof} First, assume that $e\in G'$. Then it belongs to a path $\nu = (\nu_n, \ldots, \nu_1)$ leading to a cycle. By Lemma \ref{homToFinite}, $\pi(\nu_n)=0$. Since
$$\nu_{n-1}\nu_{n-1}^* = r(\nu_{n-1}) = s(\nu_n) = \nu_n^*\nu_n, $$ we obtain
$$\pi(\nu_{n-1}\nu_{n-1}^*) = \pi(\nu_n^*\nu_n) =0.$$ So $\pi(\nu_{n-1})=0$. Continuing in the same way, we obtain $\pi(\nu_i)=0, $ for each $i$.
It follows that $\pi(e)=0$, for any $e\in G'$.

Now let $e$ with $s(e) = s(f)$, for some $f\in G'^1$, be arbitrary. By what is proved above $\pi(f)=0$. Hence $\pi(e)^*\pi(e)=\pi(f)^*\pi(f) =0$.
\end{proof}



\begin{lemma}\label{lemmaVanishingOnCycle} Let $A$ be a finite C*-algebra and $\pi: C^*(G)\to A$ a $\ast$-homomorphism. Let $c$ be a cycle which has an edge that belongs to $G'$. Then $\pi(e)=0$, for any $e\in G$ with $s(e)\in c$.
\end{lemma}
\begin{proof} Let $f\in c\bigcap G'$. By Lemma \ref{lemma3MSP} $\pi(f)=0$. Let $p_1, \ldots, p_n$ be the vertices of $c$. WLOG we can assume that $s(f)= p_n$.  Then $\pi(p_n)=0$. Since $p_{n-1}$ is equivalent to a subprojection of $p_n$, $\pi(p_{n-1})=0.$ Continuing, we obtain that $\pi$ vanishes on all the vertices of $c$. Then for any $e\in G$ with $s(e)\in c$ we have $\pi(e)^*\pi(e)=0$.
\end{proof}


\begin{theorem}\label{mainIngredient} Let $A$ be a finite C*-algebra and $\pi: C^*(G)\to A$ a $\ast$-homomorphism. Then for any edge $e\notin \tilde G$,
$\pi(e)=0$.
\end{theorem}
\begin{proof}

{\bf Claim 1}: If $e\in G\setminus \tilde G$ and $\pi(e)\neq 0$, then $s(e)\notin G'$.

\medskip

{\it Proof of Claim 1}: Suppose $s(e)\in G'$. By Lemma \ref{lemma3MSP}, $s(e)\neq s(f)$, for any $f\in G'$. Hence $s(e)$ belongs to a  cycle $c$.
If none of the edges of $c$ is in $G'$, then $e$ satisfies the condition (1) from the definition of $\tilde G$, that contradicts to the fact that $e\in G\setminus \tilde G$. If at least one of the edges of $c$ belongs to $G'$, then by Lemma \ref{lemmaVanishingOnCycle} $\pi(e)=0$. Contradiction.
 Claim is proved.

\bigskip

{\bf Claim 2}: If $e\in G\setminus \tilde G$ and $\pi(e)\neq 0$, then there exists $e_1\in G\setminus \tilde G$ such that
$$r(e_1) = s(e)$$
$$s(e_1)\neq s(e)$$
$$\pi(e_1)\neq 0.$$

\medskip

{\it Proof of Claim 2}:
By Claim 1, $s(e)\notin G'$. On the other hand, since $e$ does not satisfy the condition (2) from the definition of $\tilde G$, $e$ lies on a path starting in $G'$. Therefore in this path there is an edge preceding $e$ with source not equal to $s(e)$. Hence the set $$E:=  \{h\;|\; r(h) = s(e), \; s(h) \neq s(e)\} \neq \emptyset.$$ Since there are only finitely many vertices in the graph, if $E$ was infinite, there would be a vertex $p\neq s(e)$ and infinitely many edges $f_i$ with $s(f_i)=p$ and $r(f_i) = s(e)\notin G'$. This contradicts to the fact that $e$ does not satisfy the condition (4) from the definition of $\tilde G$. Thus $E$ is a finite non-empty set.

We observe that $s(e)$ cannot lie on a cycle, because otherwise $E\subset G'$ and therefore $s(e)\in G'$ which we know is not. Thus
$$E = \{h\;|\; r(h) = s(e)\}$$ and then we have the following Cuntz-Crieger relation:

$$s(e) = \sum_{h\in E} hh^*.$$ Then
$$0\neq \pi(e)^*\pi(e) = \pi(s(e)) = \sum_{h\in E}\pi(h)\pi(h)^*.$$
Therefore there exists $e_1\in E$ with $\pi(e_1)\neq 0$. By Lemma \ref{lemma2MSP}, $e_1\in G\setminus \tilde G$. Claim is proved.

\medskip

Now let $e\in G\setminus \tilde G$. Suppose, for the sake of contradiction, that $\pi(e)\neq 0$. We find $e_1$ as in Claim 2. For $e_1$ we find $e_2$ as in Claim 2, and so on. We obtain edges $e_i\in G\setminus \tilde G$, $i\in \mathbb N$, such that $r(e_i) = s(e_{i-1})$, $s(e_i)\neq s(e_{i-1}), \pi(e_i)\neq 0$. Since our graph has finitely many vertices, the set
$$\{s(e_i)\;|\; i\in \mathbb N\}$$ is finite, hence there exists $N$ such that
$$s(e_N) = s(e_i),$$ for some $i<N$. Then $$c: = (e_N, e_{N-1}, \ldots, e_i)$$ is a cycle. Since $\pi(e_j)\neq 0$ and $e_j\in G\setminus \tilde G$,   $j=i, \ldots, N$,  by Claim 1 $s(e_j)\notin G'$. Therefore $e_j\notin G'$, $j=i, \ldots, N$.  But then $e$ satisfies the condition (1) from the definition of $\tilde G$. This contradicts to the fact that $e\in G\setminus \tilde G$.
\end{proof}

\begin{corollary}\label{vanishingOnVertices} Let $A$ be a finite C*-algebra and $\pi: C^*(G)\to A$ a $\ast$-homomorphism. Then for any vertex $p\notin \tilde G$,
$\pi(p)=0$.
\end{corollary}
\begin{proof} If $p$ is a source of some $e\notin  \tilde G$, then $\pi(p)=0$ by previous theorem and one of the Cuntz-Krieger relations.
  
  If $p$ was a source of some $e\in \tilde G$ or if it was an isolated vertex, then $p$ would belong to $\tilde G$. 
  
  Thus we can assume that $p$ does not emit any edges  and receives some edge. Since $p\notin  \tilde G$, $p$ receives only edges of $G\setminus \tilde G$. If it receives finitely many edges, then $\pi(p)=0$ by previous theorem and another  Cuntz-Krieger relation. If it receives infinitely many edges, at least one of them belongs to $G'$ because otherwise $p$ would belong to $\tilde G$ by the item (4) in the definition of $\tilde G$. But then $p$ lies either on a path leading to a cycle  or on a cycle and therefore emits some edge   which contradicts to what we have assumed.
\end{proof}

Let $H$ be a subgraph of $G$. Similar to notations of section 3, if $p\in H$ is a vertex, then  we use notation $p$ for the corresponding projection in $C^*(G)$ and notation $\bar p$ for the corresponding projection in $C^*(H)$. Similar notation is used for edges of $H$.

\begin{lemma}\label{RestrictionUnderCondition} Let $H$ be a subgraph of a graph $G$, $B$ a C*-algebra,  and $\rho: C^*(G) \to B$ a $\ast$-homomorphism such that
$$\rho(p) = \sum_{e\in r^{-1}(p)\bigcap H} \rho(e)\rho(e)^*,$$ for each $p\in H$ such that $0< \sharp(r^{-1}(p)\bigcap H) < \infty$.  Then there exists a unique $\ast$-homomorphism $\tilde \rho: C^*(H) \to B$
such that $$\tilde \rho(\bar p) = \rho(p),\;  \tilde\rho(\bar e) = \rho(e),$$ for any $p, e\in H$.
\end{lemma}

\begin{proof} We only need to check that $$\tilde \rho(\bar p):= \rho(p),\;  \tilde\rho(\bar e) := \rho(e),$$ where $p, e\in H$, satisfy the Cuntz-Krieger relations for $C^*(H)$. It is clear that $\{\tilde \rho(\bar p)\;|\;p\in H\}$ is a collection of pairwise orthogonal projections and $\{\tilde \rho(\bar e)\;|\;e\in H\}$ is a collection of partial isometries
with mutually orthogonal ranges. Let $\bar p\in H$. For any $\bar e\in H$ with $s(e)=p$ we have $$\tilde\rho(\bar p) = \rho(p) = \rho(e)^*\rho(e) = \tilde\rho(e)^*\tilde\rho(e).$$ For any $\bar e\in H$ with $r(e)=p$ we have
$$\tilde\rho(\bar e)\tilde\rho(\bar e)^* = \rho(e)\rho(e)^*\le \rho(p) = \tilde\rho(\bar p).$$
If $0< \sharp(r^{-1}(p)\bigcap H) < \infty$, then, using our assumption on $\rho$ we obtain $$\tilde\rho(\bar p) = \rho(p) =   \sum_{e: r(e) =p, e\in H}
\rho(e)\rho(e)^* = \sum_{e: r(e) =p, e\in H} \tilde\rho(\bar e)\tilde\rho(\bar e)^*.$$
We have shown that all the Cuntz-Krieger relations involving $\bar p$  are satisfied.
\end{proof}


\begin{lemma}\label{surjection} Let $H$ be a subgraph of a graph $G$ such that $$s(e)\in H \Rightarrow e\in H$$
 and, if $p\in H$ and $0< \sharp(r^{-1}(p))<\infty$, then $r^{-1}(p)\bigcap H \neq \emptyset$. 
 Then the map $\alpha: G\to H$ defined by
$$\alpha(p) = \begin{cases} \bar p, p\in H\\ 0, p\notin H, \end{cases}\;   \alpha(e) = \begin{cases} \bar e, e\in H\\ 0, e\notin H\end{cases}$$ extends to a surjective $\ast$-homomorphism $\hat \alpha: C^*(G)\to C^*(H)$.
\end{lemma}
\begin{proof} It suffices to check that $\alpha(p), \alpha(e)$, where $p\in G, e\in G$, satisfy Cuntz-Krieger relations for $C^*(G)$. It is clear that $\{\alpha( p)\;|\;p\in G\}$ is a collection of pairwise orthogonal projections and $\{\alpha(e)\;|\;e\in G\}$ is a collection of partial isometries
with mutually orthogonal ranges. Let $p\in G$. If $p\notin H$, then all edges going to and out of $p$ are also not in $H$. Therefore all  Cuntz-Krieger relations involving $p$ are clearly satisfied. So let $p\in H$. Using our assumption on $H$, for any $e$ with $s(e)=p$ we obtain
$$\alpha(e)^*\alpha(e) = \bar e^*\bar e = \bar p = \alpha(p).$$
For any $e\in G$ with $r(e)=p$ we have
$$\alpha(e)\alpha(e)^* = \begin{cases} \bar e\bar e^*,e\in H\\ 0, e\notin H \end{cases} \le  \bar p = \alpha(p).$$
We also obtain, using our second assumption, that  $$\sum_{e: r(e) =p}\alpha(e)\alpha(e)^* = \sum_{e: r(e) =p, e\in H}\alpha(e)\alpha(e)^* = \sum_{e: r(e) =p, e\in H}\bar e\bar e^* = \bar p = \alpha(p),$$ when $0< \sharp(r^{-1}(p)) < \infty$.
\end{proof}

\begin{lemma}\label{tildeGfinite} $C^*(\tilde G)$ is matricially semiprojective if and only if $\tilde G$ is finite.
\end{lemma}
\begin{proof} Let $H_1$ be the union of all cycles of $\tilde G$ and let $H_2$ be the rest of $\tilde G$ which is necessarily a forest. By the construction of $\tilde G$,  in $\tilde G$ no cycle has an entry. Hence, by Theorems \ref{case1} and \ref{case2}, $C^*(\tilde G)$ is the amalgamated free product
$$C^*(\tilde G) = \left(C^*(H_1)(\oplus \;\text{possibly} \;\mathbb C)\right) \ast_{\mathbb C^N} \left(C^*(H_1)(\oplus \;\text{possibly}\; \mathbb C)\right),$$ where $N\in \mathbb N$. Since cycles in $\tilde G$ are disjoint, $C^*(H_1) =  C(\mathbb T)^{\oplus m}$, for some $m\in\mathbb N$.  Hence $C^*(H_1)(\oplus \;\text{possibly}\; \mathbb C)$ is semiprojective. If $\tilde G$ is finite, then $H_2$ is finite. Since $H_2$ is a forest, $C^*(H_1)(\oplus \;\text{possibly}\; \mathbb C)$ is finite-dimensional, hence semiprojective. Since semiprojectivity passes to free products amalgamated over finite-dimensional subalgebras (\cite[Prop. 2.32]{BlackadarShapeTheory}), $C^*(\tilde G)$ is semiprojective, hence matricially semiprojective.

The other way around, suppose $C^*(\tilde G)$ is matricially semiprojective. Since in $\tilde G$ no cycle has an entry, $C^*(\tilde G)$ is quasidiagonal by \cite{Schafhauser}. Matricial semiprojectivity implies then that $C^*(\tilde G)$ is RFD.  By Theorem \ref{RFD} $\tilde G$ is finite.
\end{proof}

\begin{lemma}\label{fixingMistake} $H= \tilde G$ satisfies the conditions of Lemma \ref{surjection}.
\end{lemma}
\begin{proof} It follows from definition of $\tilde G$ that
$$s(e)\in \tilde G \Rightarrow e\in \tilde G.$$ Let $p\in \tilde G$ and $0< \sharp(r^{-1}(p))<\infty$. We need to prove that $p$ is the range of some edge of $\tilde G$. Since $p\in \tilde G$ and $p$ is not an infinite receiver and not an isolated vertex, $p$ is  a vertex of some edge $e\in \tilde G$ (see definition of $\tilde G^{(0)}$ on p.15).   

If $p= r(e)$, we are done.

So we can assume $p= s(e)$. Since $e\in \tilde G$, $e$ satisfies one of the conditions (1)-(4). 

Suppose $e$ satisfies (1). Then it can be reached from a cycle $c$ all of whose edges are not in $G'$. Then this cycle and each vertex of the path starting at $c$ and leading to $e$ are in $\tilde G$. In particular, the edge of this path that enters $e$ (possibly it is an edge of the cycle) belongs to $\tilde G$ and has range $p$, so we are done.

Suppose $e$ satisfies (2). Then any $f\in r^{-1}(p)\neq \emptyset$ also satisfies (2) and therefore $f\in \tilde G$, so we are done. 

Suppose $e$ satisfies (3). Then there is a path starting at a vertex of an edge satisfying (2) and leading to $e$. The edge of this path that enters $e$ satisfies either (3) or (2), hence belongs to $\tilde G$, so we are done.

It remains to notice that $e$ does not satisfy the condition (4) by the assumption. 
\end{proof}

\begin{theorem}\label{MSP1} If $C^*(G)$ is matricially semiprojective, then $\tilde G$ is finite.
\end{theorem}
\begin{proof} We are going to prove that if $C^*(G)$ is matricially semiprojective, then $C^*(\tilde G)$ is matricially semiprojective.
Then Lemma \ref{tildeGfinite} will finish the proof.

Let $\rho: C^*(\tilde G) \to \prod M_n/\oplus M_n$ be a $\ast$-homomorphism. 
 By Lemma \ref{fixingMistake}, $\tilde G$ satisfies the condition of Lemma \ref{surjection}. Let $\alpha: C^*(G)\to C^*(\tilde G)$ be as in Lemma \ref{surjection}. Since $C^*(G)$ is matricially semiprojective, $\rho\circ\alpha$ lifts.
 
\bigskip

{\bf Claim}: There exists a lift  $\psi: C^*(G) \to \prod M_n$ of $\rho\circ\alpha$ such that $$\psi(p) = \sum_{e\in r^{-1}(p)\bigcap \tilde G} \psi(e)\psi(e)^*,$$ for each $p\in \tilde G$ with $0< \sharp(r^{-1}(p)\bigcap \tilde G) < \infty$.

\bigskip

{\it Proof of Claim}: Let a $\ast$-homomorphism $\psi:C^*(G) \to \prod M_n$ be some lift of $\rho\circ\alpha$. Then for any vertex $p\in \tilde G$ with $0< \sharp(r^{-1}(p)\bigcap \tilde G) < \infty$, $\psi(p) - \sum_{e\in r^{-1}(p)\bigcap \tilde G} \psi(e)\psi(e)^*$ is a projection. Its image under the quotient map $\pi: \prod M_n \to \prod M_n/\oplus M_n$ is $$\rho\circ\alpha(p) - \sum_{e\in r^{-1}(p)\bigcap \tilde G} \rho\circ \alpha(e)\rho\circ\alpha(e)^* =
\rho(\bar p) - \sum_{e\in r^{-1}(p)\bigcap \tilde G} \rho(\bar e)\rho(\bar e)^* =0.$$ Therefore $$\psi(p) - \sum_{e\in r^{-1}(p)\bigcap \tilde G} \psi(e)\psi(e)^*\in \oplus M_n.$$ Any projection in $\oplus M_n$ has all but finitely many coordinates equal zero. So there is $N(p)\in \mathbb N$  such that all but the first $N(p)$ coordinates of the projection $\psi(p) - \sum_{e\in r^{-1}(p)\bigcap \tilde G} \psi(e)\psi(e)^*$ are zeros. Since $G$ has finitely many vertices, so does $\tilde G$. Replacing coordinates of $\psi$ by zero maps on the first  $\max_{p\in \tilde G, 0< \sharp(r^{-1}(p)\bigcap \tilde G) < \infty}N(p)$ places,   we obtain a $\ast$-homomorphism, which we again denote by $\psi$, that is  still a  lift of $\rho\circ \alpha$ and satisfies the condition $$\psi(p) = \sum_{e\in r^{-1}(p)\bigcap \tilde G} \psi(e)\psi(e)^*,$$ for each $p\in \tilde G$ such that $0< \sharp(r^{-1}(p)\bigcap \tilde G) < \infty$. Claim is proved.

\bigskip

By Lemma \ref{RestrictionUnderCondition}, there is a $\ast$-homomorphism $\tilde \psi: C^*(\tilde G) \to \prod M_n$
such that $$\tilde \psi(\bar p) = \psi(p),\;  \tilde\psi(\bar e) = \psi(e),$$ for any $p, e\in \tilde G$.
Since $\prod M_n$ is a finite C*-algebra, by Theorem \ref{mainIngredient} and Corollary \ref{vanishingOnVertices} the $\ast$-homomorphism $\psi$ vanishes on all $p, e\notin \tilde G$. It implies that $\psi$ and $\tilde\psi\circ\alpha$ coincide on all the generators of $G$. Therefore $$\psi=\tilde\psi\circ\alpha.$$ Thus $\pi\circ \tilde\psi\circ\alpha =\rho\circ\alpha$. Hence $\pi\circ\tilde\psi$ and $\rho$ coincide on the image of $\alpha$. Since $\alpha$ is surjective, we conclude that $\pi\circ\tilde\psi =\rho$. We have shown that an arbitrary  $\ast$-homomorphism $\rho: C^*(\tilde G) \to \prod M_n/\oplus M_n$ is liftable. Thus $C^*(\tilde G)$ is matricially semiprojective.
\end{proof}

\bigskip

If there exists infinitely many edges with source in vertex $p$ and range in a vertex $q$, then we will say that between $p$ and $q$ there is an {\it edge of infinite multiplicity}.

\bigskip

\begin{lemma}\label{baseOfInduction} Suppose $\tilde G$ is finite and each vertex $p\in \tilde G$ satisfying $0< \sharp(r^{-1}(p)\bigcap \tilde G) < \infty $ is either 1) not the range of an edge of infinite multiplicity belonging to $G\setminus \tilde G$ (that is there is no $q\in G$  such that the collection $\{e\;|\; s(e)=q, r(e)=p, e\notin \tilde G\}$ is infinite)
or 2) the range of an edge of infinite multiplicity belonging to $G'$ (that is there is  $q\in G$  such that the collection $\{e\;|\; s(e)=q, r(e)=p, e\in G'\}$ is infinite). Then $C^*(G)$ is matricially semiprojective.
\end{lemma}
\begin{proof} Let $\rho: C^*(G)\to \prod M_n/\oplus M_n$. Let $p\in \tilde G$ be such that $$0< \sharp(r^{-1}(p)\bigcap \tilde G) < \infty.$$
 If $p$ is the range of an edge (of infinite multiplicity) belonging to $G'$, then  $p$ belongs to both $G'$ and $\tilde G$ and therefore must be a  vertex of a cycle. Then any edge with range $p$ lies in $G'$. Then, by Lemma \ref{notInTildeG}, $p$ does not receive any edges of $\tilde G$ which contradicts to the fact that $\sharp(r^{-1}(p)\bigcap \tilde G)>0$. Therefore $p$ cannot be the range of an edge (of infinite multiplicity) belonging to $G'$ and then, by the assumption,  $p$ is not the range of an edge of infinite multiplicity belonging to $G\setminus \tilde G$.

  Then  there exists only  finitely many edges of $G\setminus \tilde G$  with range in $p$. Hence
$$0< \sharp(r^{-1}(p)) < \infty.$$ Then $$p= \sum_{e: r(e)=p, e\in G} ee^*,$$ and therefore
$$\rho(p) =  \sum_{e: r(e)=p, e\in G} \rho(e)\rho(e)^*.$$
By Lemma \ref{mainIngredient}, $\rho(e)=0$, for each $e\in G\setminus \tilde G$. Hence
$$\rho(p) =  \sum_{e: r(e)=p, e\in \tilde G} \rho(e)\rho(e)^*.$$
By Lemma \ref{RestrictionUnderCondition} there exists a unique $\ast$-homomorphism $\tilde \rho: C^*(\tilde G)\to \prod M_n/\oplus M_n$ such that
$$\tilde\rho(\bar p) =\rho(p), \; \tilde \rho(\bar e)= \rho(e),$$
for $\bar p, \bar e\in \tilde G.$   By Lemma \ref{fixingMistake}, $\tilde G$ satisfies the condition of Lemma \ref{surjection}. Let $\alpha: C^*(G)\to C^*(\tilde G)$ be as in Lemma \ref{surjection}. Then for any $e\in G$,

\begin{multline*}\tilde\rho\circ\alpha(e) = \begin{cases} \tilde\rho(\bar e), e\in \tilde G\\ 0, e\notin \tilde G \end{cases} =
\begin{cases} \rho(e), e\in \tilde G\\ 0, e\notin \tilde G \end{cases} = ^{Lemma\ref{mainIngredient}} \rho(e).\end{multline*}
Similarly, using Corollary \ref{vanishingOnVertices} we obtain
$$\tilde\rho\circ\alpha(p) = p, $$ for any vertex $p\in G$. Thus $$\tilde\rho\circ \alpha = \rho.$$
Since $\tilde G$ is finite, by Lemma \ref{tildeGfinite} $C^*(\tilde G$) is matricially semiprojective. Hence $\tilde \rho$ lifts to some $\ast$-homomorphism $\psi: C^*(\tilde G)\to \prod M_n/\oplus M_n$. Then $\psi\circ\alpha$ is a lift of $\rho$.
\end{proof}

\begin{theorem}\label{ifTildeGfinite}  If $\tilde G$ is finite, then $C^*(G)$ is matricially semiprojective.
\end{theorem}
\begin{proof} We use induction on the number of vertices $p$ satisfying the following three conditions:

\begin{enumerate}[label=(\alph*)]

\item $0< \sharp(r^{-1}(p)\bigcap \tilde G) < \infty,$

\item $p$ is the range of an edge of infinite multiplicity belonging to $G\setminus \tilde G$

\item $p$ is not the range of an edge of infinite multiplicity belonging to $G'$.

\end{enumerate}

\medskip

\noindent When the number of such vertices is zero, the statement is proved in Lemma \ref{baseOfInduction}.  Suppose it is proved for $k\le N$. We will prove it for $k=N+1$. So suppose $G$ has $N+1$ vertices satisfying the conditions (a)-(c) and let $p$ be one of these vertices.

Let $H$ be the graph obtained from $G$ by removing the set of edges $\{e\in \tilde G\;|\; r(e)=p\}.$

\bigskip

{\bf Claim 1:} $H'=G'$ and $\tilde H\subset \tilde G$.

\medskip

{\it Proof of Claim 1}: We have $H'\subset G'$. To show the opposite implication, let $e\in G$ be an edge that leads to a cycle $c\subset G$. We need to prove that $e\in H, c\subset H$. If $e\notin H$, then $r(e)= p$. Then from $p$ one can reach $c$ and therefore all edges with range $p$ must belong to $G'$. This contradicts to the fact that $p$ satisfies the conditions (b) and (c). Thus $e\in H$.
Suppose  $c\not\subset H$. Then there is $f\in c$ such that $r(f)=p$ (and $f\in \tilde G$). Then $p\in c$. Hence an edge of infinite multiplicity with range $p$ that exists by the condition (b) must belong to $G'$ which contradicts to (c). So $c\subset H$. We have proved that $H'=G'$.

Let $e\in \tilde H$. We need to prove that $e\in \tilde  G$. If $e$ satisfies the condition (1) from the definition of $\tilde H$, then, since $H'=G'$, it satisfies the condition (1) from the definition of $\tilde G$, so $e\in \tilde G$.

Suppose $e$ satisfies the condition (2) from the definition of $\tilde H$. If it satisfies the condition (2) from the definition of $\tilde G$, then $e\in \tilde G$. If it does not, then $e$ can be reached from $G'=H'$ and the corresponding path contains an edge $f\notin H$. Then $f\in \tilde G$. By Lemma \ref{lemma2MSP} all edges that follow $f$, in particular $e$, belong to $\tilde G$.

Suppose $e$ satisfies the condition (3) from the definition of $\tilde H$ which means that it can be reached (inside $H$) from a vertex of an edge $f\in H$ satisfying the condition (2) from the definition of $\tilde H$. By what is proved above $f\in \tilde G$. By Lemma \ref{lemma2MSP} $e\in \tilde G$.

If $e$ satisfies the condition (4) from the definition of $\tilde H$, then, since $H'=G'$, it satisfies the condition (4) from the definition of $\tilde G$, so $e\in \tilde G$. Claim is proved.

\bigskip

\noindent We conclude from Claim 1 that $\tilde H$ is finite. By the construction of $H$, $H$ has $N$ vertices satisfying the conditions (1)-(3). By the induction assumption, $C^*(H)$ is matricially semiprojective.

Let $\rho: C^*(G) \to \prod M_n/\oplus M_n$ be a $\ast$-homomorphism.  For any vertex $q\in H$ with $\sharp(r^{-1}(q)\bigcap H) < \infty$ we have
 $q\neq p$. Hence for such $q$, $\sharp(r^{-1}(q)\bigcap H) = \sharp(r^{-1}(q)).$ Then  for $q\in H$ with $0< \sharp(r^{-1}(q)\bigcap H) < \infty$ we have $0<\sharp(r^{-1}(q))<\infty$ and then
 $$\rho(q) = \sum_{e\in r^{-1}(q)} \rho(e)\rho(e)^* = \sum_{e\in r^{-1}(q)\bigcap H} \rho(e)\rho(e)^*.$$ By Lemma \ref{RestrictionUnderCondition} there is  $\ast$-homomorphism $\tilde \rho:C^*(H)\to \prod M_n/\oplus M_n$ such that $$\tilde\rho(\bar e) = \rho(e), \tilde\rho(\bar p) = \rho(p),$$ for $e, p\in H$. Since $C^*(H)$ is matricially semiprojective, $\tilde\rho$ lifts to some $\ast$-homomorphism $\tilde\psi: C^*(H)\to \prod M_n$.

 Let us denote the edges of $\tilde G$ that enter $p$ (that is, exactly those edges that were removed when we defined $H$) by $f_1, \ldots, f_m$.

 \bigskip

 {\bf Claim 2}: The projections $\rho(f_1)\rho(f_1)^*,  \ldots, \rho(f_m)\rho(f_m)^*$ lift to pairwisely orthogonal subprojections $R_1, \ldots, R_m$ of $\tilde\psi(\bar p)$ such that   $R_i$ is orthogonal to $\tilde\psi(\bar e)$, for each $i=1, \ldots, m$ and each $e\notin \tilde G$ with $r(e) =p$.

\medskip

{\it Proof of Claim 2}: Standard. For elements $a$ of $\prod M_k$ we will  use notation $a= (a_k)_{k\in \mathbb N}$.  The set
$\{\tilde \psi(\bar e)\tilde\psi(\bar e)^*\;|\; e\notin \tilde G, r(e)=p\}$ is an infinite set of projections in $\prod M_k$.  Since in $M_k$ there are only finitely many pairwise orthogonal projections, the sum $\sum_{e\notin \tilde G, r(e)=p} \left(\tilde \psi(\bar e)\tilde\psi(\bar e)^*\right)_k$ is finite, for each $k\in \mathbb N$.  We denote by $\sum_{e\notin \tilde G, r(e)=p} \tilde \psi(\bar e)\tilde\psi(\bar e)^*$ the projection
$$\sum_{e\notin \tilde G, r(e)=p} \tilde \psi(\bar e)\tilde\psi(\bar e)^*: = \left(\sum_{e\notin \tilde G, r(e)=p} \left(\tilde \psi(\bar e)\tilde\psi(\bar e)^*\right)_k\right)_{k\in \mathbb N}.$$ Let
$$Q= (Q_k)_{k\in\mathbb N}:= \tilde\psi(\bar p) - \sum_{e\notin \tilde G, r(e)=p} \tilde \psi(\bar e)\tilde\psi(\bar e)^*.$$ Then
$$\rho(f_i)\rho(f_i)^* \in \prod Q_kM_kQ_k/\oplus Q_kM_kQ_k,$$ $i=1, \ldots, m$, are pairwise orthogonal projections.  Since $\mathbb C^m$ is matricially semiprojective,   $\rho(f_i)\rho(f_i)^*,$ $i=1, \ldots, m$, lift to pairwise orthogonal projections $R_i\in \prod Q_kM_kQ_k$, $i=1, \ldots, m$.  Then each $R_i$ is orthogonal to $\tilde\psi(\bar e)\tilde\psi(\bar e)^*,$ for $e\notin \tilde G$ with $r(e)=p$, and $R_i\le Q\le \tilde\psi(\bar p)$. Claim 2 is proved.

\bigskip

By \cite[Prop. 2.23]{BlackadarShapeTheory} $\rho(f_i)$ lifts to a partial isometry $u_i\in \prod M_k$ such that
$$u_i^*u_i = \tilde\psi(\bar f_i^*\bar f_i), \;\; u_iu_i^* = R_i,$$ $i=1, \ldots, m$.  Since $p$ is the range of an edge of infinite multiplicity, the CK-relations in $G$ involving $p$ are

\medskip

1) $f_if_i^*\le p, \;i=1, \ldots, m$,

2) $ee^*\le p$, for $e\notin\{f_1, \ldots, f_m\}$ with $r(e)=p$,

3) $e^*e=p, $ for $e$ with $s(e)=p$.

\medskip

We note that none of $f_i$'s  is a loop because otherwise the assumptions $(b)-(c)$ on $p$ would not hold.  Hence the relation 3) does not involve $f_i$'s.  Therefore the correspondence
$$f_i\mapsto u_i, $$
$$e\mapsto \tilde\psi(\bar e), $$
where $\;i=1, \ldots, m,$ $e\notin \{f_1, \ldots, f_m\}$, defines a $\ast$-homomorphism $\psi: C^*(G)\to \prod M_k$ which is a lift of $\rho$.
\end{proof}

From Theorem \ref{MSP1}  and Theorem \ref{ifTildeGfinite} we obtain

\begin{theorem} $C^*(G)$ is matricially semiprojective if and only if $\tilde G$ is finite.
\end{theorem}

\end{document}